\documentclass[review,onefignum,onetabnum]{siamart190516}

\def\<{\langle}
\def\>{\rangle}

\definecolor{vert_10}{RGB}{21,106,47}

\def\dom{\textnormal{Dom}\,}

\usepackage{lipsum}
\usepackage{amsfonts}
\usepackage{graphicx}
\usepackage{amssymb} 
\usepackage{epstopdf}
\usepackage{algorithmic}
\ifpdf
  \DeclareGraphicsExtensions{.eps,.pdf,.png,.jpg}
\else
  \DeclareGraphicsExtensions{.eps}
\fi

\newsiamremark{remark}{Remark}
\newsiamremark{hypothesis}{Hypothesis}
\crefname{hypothesis}{Hypothesis}{Hypotheses}
\newsiamthm{claim}{Claim}
\newsiamremark{example}{Example}

\newtheorem{theo}{Theorem}
\newtheorem{defi}{Definition}
\newtheorem{rem}{Remark}

\newtheorem{expl}{Example}
\newtheorem{propo}{Proposition}
\newtheorem{lem}{Lemma}
\newtheorem{coro}{Corollary}
\allowdisplaybreaks

\def\dom{\mathop{\rm dom\,}}
\headers{An inertial proximal method for bilevel equilibria}{A.Balhag,  Z. Mazgouri , M. Th\'era}

\title{Weak and strong convergence of an inertial proximal
method for solving bilevel monotone equilibrium problems\thanks{Submitted to the editors DATE.
\funding{Research of A\"icha Balhag was supported by  the EIPHI Graduate School (contract ANR-17-EURE-0002)  and 
research of Michel Th\'era was supported by a public grant as part of the  Investissement d'avenir project, reference ANR-11-LABX-0056-LMH, LabEx LMH.}}}
\author{A. Balhag\thanks{Institut de Math\'ematiques de Bourgogne, UMR 5584 CNRS, Universit\'e Bourgogne Franche-Comt\'e, F-2100 Dijon, France   (\email{aichabalhag@gmail.com}).}
\and Z. Mazgouri \thanks{ORCID 0000-0003-0131-4741, National School of Applied Sciences, Sidi Mohamed Ben Abdellah University, Laboratory LAMA Mathematics,
	30000 Fez, Morocco (\email{zakariam511@gmail.com}).}
\and Michel Th\'era\thanks{ORCID 0000-0001-9022-6406, Laboratoire XLIM UMR-CNRS 7252, Universit\'e de Limoges, 87032 Limoges, France and Federation University Australia, Ballarat 3353, Australia \email{(michel.thera@unilim.fr}).
}}

\usepackage{amsopn}

\newcommand{\R}{\mathbb R}

\def\dom{\mathop{\rm dom\,}}

\def\argmin{\mathop{\rm argmin\,}}

\ifpdf
\fi
\usepackage{caption}
\begin{document}
\maketitle
\begin{abstract}
 In this paper, we introduce an inertial proximal method for solving a bilevel problem involving two monotone equilibrium bifunctions in Hilbert spaces.
	Under suitable conditions and without any restrictive assumption on the trajectories, the weak and strong convergence of the sequence generated by the iterative method are established. 
	Two particular cases illustrating the proposed method are thereafter discussed with respect to hierarchical minimization problems and equilibrium problems under saddle point constraint.
	Furthermore, a numerical example is given to demonstrate the implementability of our algorithm. The algorithm and its convergence results improve and develop previous results in the field.  
\end{abstract}


% REQUIRED

%%%%%%%%%%%%%%%%%%%%%%%%%%%%%%%%%%%%%%%%%
%%%%%%%%%%%%%%%%%%%%%%%%%%%%%%%%%%%%%%%%%%
Let $K$ be a nonempty closed and convex subset of a real Hilbert space $H$, and let $f:K\times K \rightarrow \mathbb{R}$ be a real-valued bifunction. %satisfying $f(x,x)=0$ for all $x\in K$.
%Given a real Hilbert space $H$ with inner product $\langle \cdot ,\cdot \rangle$ and induced norm $\|\cdot\|$, a nonempty closed convex set $K\subset H$ and a real-valued bifunction $f:K\times K \rightarrow \mathbb{R}$ satisfying $f(x,x)=0$ for all $x\in K$, 
The equilibrium problem 
%in the sense of Blum and Oettli 
\cite{BO} associated with the bifunction $f$ on $K$ is stated as follows: 
find $\bar{x}\in K$ such that
\begin{equation*}\label{ep}
\tag{$EP$} 
%\text{Find} \;\; \bar{x}\in K \;\; \text{such that}  \;\;\; 
f(\bar{x},y)\geq 0, \;\; \forall y\in K.
\end{equation*}
This abstract variational formulation constitutes a  
convenient unified mathematical model for many problems in applied mathematics such as optimization problems, variational and hemivariational inequalities, fixed-point and saddle point problems,  network equilibrium problems, Nash equilibrium and others, see for instance \cite{BO,bns,CCR,mosco} and the bibliography therein.

\smallskip
One of the most popular algorithm for solving $(EP)$ is the proximal point method $(PPM)$ extended from variational inequalities to equilibrium problems by Moudafi \cite{moud3}.
%Further efforts have then been dedicated to iterative methods for solving \eqref{ep}, let us quote in this sense the paper by Moudafi \cite{moud1} who extended the proximal point method $(PPM)$ from variational inequalities to equilibrium problems.
In this regard, by introducing the resolvent of the bifunction $f$ (see, \cite{BO}), defined, for $\lambda >0$, by $$ J_\lambda^f(x):=\{z\in K: f(z,y)+\frac{1}{\lambda}\langle z-x,y-z\rangle\geq 0, \; \forall y\in K\},$$
%inspired by the resolvent of bifunctions which appeared implicitly in \cite{BO}, 
the author in \cite{moud3} suggested a basic version of $(PPM)$ for $(EP)$ in the monotone framework. This method generates the next iterate $x_{n+1}$, for each $n\geq 0$, by solving the following subproblem $x_{n+1}=J_{r_n}^f(x_n)$, i.e.,
\begin{equation}\label{PPM-EP}
f(x_{n+1} , y) + \frac{1}{r_n}\langle x_{n+1} - x_{n} , y - x_{n+1}\rangle \geq 0,\;\; \forall y \in K,
\end{equation}
where $\lbrace r_n \rbrace$ is a sequence of nonnegative numbers. \\
%the above scheme can be written as $x_{n+1}=J_\lambda^f(x_n)$.
\noindent Under the monotonicity condition on $f,$ Moudafi proved the weak convergence of the sequence $\{x_n\}$ generated by
algorithm \eqref{PPM-EP} to a solution of $(EP)$. Thereby, a great interest has been brought to the study of $(EP)$ by means of splitting %iterative methods that include gradient-descent (or forward) methods and 
proximal point (or backward) methods; one can consult \cite{ant,bk,moud-thera} and the references therein.\\

 Given its growing interest in applications to different applied domains, the problem $(EP)$  is currently considered  as one of the important  research directions in which the optimization community is interested.  Indeed, the study of the existence of a solution to this problem still falls within the scope of very recent studies concerning new methods of resolution. Let us quote in this sense the paper \cite{Cotrina}  
%and its references therein, 
in which the authors, by using the celebrated  \textit{Ekeland variational principle} under a weaker notion of continuity and without any convexity assumptions,  studied the existence of equilibria and quasi-equilibria in the setting of metric spaces. %See also its references therein wherein new concepts of equilibrium are introduced under new modes of relaxed 
%generalized 
%quasimonotonicity and quasiconvexity}.  
The bibliography of this article refers to new equilibrium concepts in which quasi-monotonicity and quasi-convexity are relaxed.\\
In this paper, we study the problem $(EP)$ in a general framework where we focus our interest on the following bilevel equilibrium problem: find $\bar{x}\in S_f$ such that  
\begin{equation*}\label{BEP}
\tag{$BEP$}%\text{Find} \;\; \bar{x}\in S_f \;\; \text{such that}
% \;\;\; 
 g(\bar{x},y)\geq 0, \;\; \forall y\in S_f,
\end{equation*}
where $g:K\times K\rightarrow \R$ is another real-valued bifunction and $S_f$ stands for the set of constraints defined by solutions to the second level equilibrium problem $(EP)$ given by $S_f:=\lbrace u\in K : f(u,y)\geq0 \;\; \forall y\in K\rbrace.$ 
We denote by $S$ the set of solutions to $(BEP)$ which we assume to be nonempty. \\ 

The problem $(BEP)$ was implicitly introduced in the paper by Chadli, Chbani and Riahi \cite{CCR} in the setting of the so-called \textit{viscosity principle} for equilibrium problems. This principle aims at a good selection of the upper equilibrium among solutions to the lower
level equilibrium problem. This class of hierarchical problems covers in both levels, all the cases cited previously for an equilibrium problem. 
%and also encompasses potential applications in a variety of research fields as transportation modelling, economics, ecology, engineering and others
Besides their unification aspect, bilevel equilibrium problems has proved over the past two decades, very good applicability in different fields covering mechanics, engineering sciences and economy, see \cite{Dempe} and references therein. 
%it is currently considered as one of the fashion research directions that interest the optimization and equilibrium community. 
Greater attention was then paid to this class of problems regarding the existence of solutions via dynamical and algorithmic approaches and also from the point of view of parametric stability. The interested reader can consult the following recent investigations \cite{ait2021dynamical, bento-2016, CMR, CR1, moud2, Thuy-2017} and the references therein.
In recent years, algorithmic resolution procedures have been widely studied
for solving $(BEP)$. Moudafi \cite{moud2} introduced, by using the penalty method \cite{CCR}, the regularized proximal point method $(RPPM)$ for solving $(BEP)$. This algorithm is described as follows: from a starting point $x_0\in K$, for each $n\geq 0$, the next iterate $x_{n+1}$ is defined 
by the proximal iteration $x_{n+1}:=J_{\lambda_n}^{f+\beta_ng} (x_n)$, i.e.,
%$x_{n+1}=\text{prox}_{\lambda_n(f+\beta_n g)}(x_n)$ 
\begin{equation}\label{PPM-BEP}
%\tag{$PPM$}\;\;\;
f(x_{n+1},y)+\beta_n g(x_{n+1},y)+\dfrac{1}{\lambda_n}\langle x_{n+1}-x_n, y-x_{n+1} \rangle \geq 0, \;\; \forall y\in K, 
\end{equation}	
where $\lbrace \beta_n \rbrace$ and $\lbrace \lambda_n \rbrace$ are two sequences of nonnegative reals.
More precisely, under suitable assumptions on the bifunctions $f$ and $g$, Moudafi  proved that the sequence $\lbrace x_n \rbrace$ generated by algorithm \eqref{PPM-BEP} converges weakly to a solution of $(BEP)$ provided that $$\displaystyle \liminf_{n\rightarrow +\infty} \lambda_n>0, \displaystyle \sum_{n=0}^{+\infty} \lambda_n \beta_n <+\infty \;\text{ and } \;\|x_{n+1}-x_n\|=o(\beta_n).$$ 
The drawback of the last assumption is the difficulty to choose such a control sequence $(\beta_n)$ because we do not know the convergence rate of $\|x _{n+1} - x_ n \|$. The author in \cite{moud2} conjectured that this restrictive assumption can be removed via the introduction of a conditioning notion for equilibrium bifunctions.

Later on, the authors in \cite{CR1} considered an alternate proximal scheme, which generates the next iterates $x_{n+1}$ by solving
the regularized problem $x_{n+1}:=J_{\lambda_n}^{\beta_nf+g} (x_n)$, i.e., $$\beta_n f(x_{n+1},y)+g(x_{n+1},y)+\dfrac{1}{\lambda_n}\langle x_{n+1}-x_n, y-x_{n+1} \rangle, \;\; \forall y\in K.$$ %where $\lbrace \lambda_n \rbrace$ and $\lbrace \beta_n \rbrace$ are positive parameters.
Here, the difficulty of the method $(RPPM)$ mentioned in \cite{moud2} has been solved. Following \cite{attouch2011prox} and under %the same 
a similar geometric assumption formulated in terms of the Fenchel conjugate function associated to the bifunction $f$, %but 
%and by reinforcing hypothesis $\displaystyle \sum_{n=0}^{+\infty}\lambda_n=+\infty$ by $\displaystyle \liminf_{n\rightarrow +\infty}\lambda_n>0$, 
they analyzed both the
weak and strong convergence of their algorithm %not the average but 
to a solution of $(BEP)$. %In particular, the geometric assumption \eqref{fitz-discret} shows that, as conjectured in \cite{moud2}, the restrictive assumption $\|x_{n+1}-x_n\|=o(\epsilon_n)$ can be removed via the introduction of a conditioning notion for equilibrium bifunctions.
More recently, in \cite{clr}, the authors proposed a forward-forward algorithm and a forward-backward algorithm for solving $(BEP)$ under quite mild conditions where the bifunction of the two level equilibrium problems are supposed pseudomonotone. \\

As a continuity of the studies of equilibrium problems by means of proximal iterative methods, we propose an inertial proximal method for solving $(BEP)$.
It is well known that the inertial proximal iteration, where the next iterate is defined by making use of the previous two iterates, may be interpreted as an implicit discretization of differential systems of second order in time. The presence of inertial terms improves the convergence behavior of the generated sequences. We emphasize that the origin of these methods dates back to \cite{AA} as part of the approach to a solution of 
an abstract inclusion of the form: find $\bar{x}\in H$ such that 
\begin{equation}\label{inclu-prob}
0\in A(\bar{x}),
\end{equation}
where $A: H\rightrightarrows H$ is a maximally monotone operator and the solution set $A^{-1}(\{0\})$ is assumed to be nonempty.
In this regard, giving two sequences of nonnegative numbers $\lbrace \alpha_n \rbrace$ and $\lbrace \lambda_n \rbrace$, the authors in \cite{AA} considered the following iterative scheme:
\begin{equation*}\label{origin-ipa}
x_{n+1}-x_n-\alpha_n(x_n-x_{n-1})+\lambda_nA(x_{n+1}) \ni 0,
\end{equation*}
and proved the weak convergence of the sequence $\{x_n\}$ generated by the above algorithm towards a solution of \eqref{inclu-prob} under appropriate conditions on the parameters $\lbrace \alpha_n \rbrace$ and $\lbrace \lambda_n \rbrace$ whenever the restrictive assumption $\displaystyle \sum_{n=1}^{+\infty} \alpha_n\| x_n-x_{n-1} \|^2<+\infty$ holds.
\vskip 2mm

Inspired by the results presented in \cite{moud1} 
in the framework of solving $(EP)$ by an approximate second order differential proximal procedure, 
and also illuminated by the results explored in \cite{CR1,moud2}, we propose a new approximate inertial proximal scheme to solve $(BEP)$: \\

\hrule
\vspace{2mm}
\noindent {\bf Algorithm:} $($Inertial proximal algorithm $(IPA))$.
\vspace{2mm}
\hrule
\vspace{2mm}
\noindent {\bf Initialization:} Choose positive sequences $\lbrace\beta_n \rbrace$, $\lbrace \lambda_n \rbrace$, and a nonnegative real number $\alpha \in [0, 1]$. Take arbitrary $x_0, x_1\in K$. 
\vspace{2mm}
\hrule
\vspace{2mm}
\noindent {\bf Iterative step:} %\begin{itemize}
	%\item 
	For every $n\geq 1$ and given current iterates $x_{n-1},  x_n\in K$ set $y_n:=x_n+\alpha(x_n-x_{n-1})$ and define $x_{n+1}\in K$ by 
	$x_{n+1}:=J_{\lambda_n}^{\beta_nf+g} (y_n),$
	%$
	%x_{n+1}=\text{prox}_{\lambda_n (\beta_nf+g)} (y_n) 
	%$
	i.e., 
	\begin{equation}\label{algo}
	\beta_n f(x_{n+1},y)+g(x_{n+1},y)+\frac{1}{\lambda_n}\langle x_{n+1}-y_n, y-x_{n+1}\rangle \geq 0, \;\;\forall y\in K.
	\end{equation}
%	\item If $x_{n+1}=y_n$ then stop and $x_{n+1}$ is the solution of (BEP).
%\end{itemize}
%where $y_n:=x_n+\alpha(x_n-x_{n-1})$, $\alpha$ is a nonnegative real number and $\lbrace\beta_n \rbrace$, $\lbrace \lambda_n \rbrace$ are sequences of positive numbers.
%where $y_n:=x_n+\alpha(x_n-x_{n-1})$.
%\vspace{2mm}
\hrule
\vspace{2mm}
In the above algorithm, $\lbrace \lambda_n \rbrace$ denotes the sequence of step sizes, $\lbrace \beta_n \rbrace$ the sequence of penalization parameters, and $\alpha\in [0, 1]$ the parameter that controls the inertial terms. The proposed numerical scheme recovers, when $\alpha=0$, the algorithm investigated in \cite{CR1}, and if in addition $f=0$, the one suggested in \cite{moud3}. The Fitzpatrick transform of the bifunction $f$ will be a key ingredient in our convergence analysis. Indeed, we 
provide conditions under which the sequence generated by the algorithm $(IPA)$ weakly or strongly converges to a solution of $(BEP)$. 
More precisely, under a discrete counterparts \eqref{fitz-discret} of the geometric condition used in \cite{CMR}
and formulated in terms of the  Fitzpatrick transform of the bifunction $f$, we first prove that (see Theorem \ref{thm-weak}) the sequence generated by 
%the above scheme 
$(IPA)$ weakly converges to a solution of $(BEP)$ provided that $0\leq \alpha< \frac{1}{3}$, $\displaystyle \liminf_{n\rightarrow +\infty}\lambda_n>0$ and $\displaystyle \lim_{n\rightarrow +\infty}\beta_n =0$. Afterwards, by strengthening the monotonicity assumption on the upper level bifunction $g$, and whenever $0\leq \alpha< \frac{1}{3}$ and $\displaystyle \sum_{n=1}^{\infty}\lambda_n=+\infty$, we show (see Theorem \ref{strong1}) the strong convergence of the trajectories generated by the proposed algorithm to the unique solution of $(BEP)$. 
Then, we show (see Theorem \ref{strong2}) that, without the need of the geometric assumption \eqref{fitz-discret}, the sequence converges strongly to the unique solution of $(BEP)$ when the parameters $\lambda_n$ and $\beta_n$ satisfy additionally conditions $\displaystyle \lim_{n\rightarrow +\infty}\lambda_n =0$, $\displaystyle \lim_{n\rightarrow +\infty}\beta_{n}=+\infty$ and $\displaystyle \liminf_{n\rightarrow +\infty}\lambda_n\beta_n>0$. 
The main advantage of our approach is that it provides convergence without any restrictive assumption on the trajectories. The results can be seen as an extension to the second order counterparts of the ones given in \cite{CR1, moud2}. To our knowledge, such inertial proximal schemes have been studied only for the first level equilibrium problem $(EP)$, see for instance \cite{j,Hieu-Gibali} and the references therein. 
%and still be not investigated for the two level problem $(BEP)$. 
As applications, we discuss the hierarchical convex minimization case and equilibrium problems under a saddle point constraint. Numerical experiment is thereafter given to illustrate our theoretical results. We end the paper by concluding comments.\\

\section{Background and technical lemmata}
\label{sec:1}
In this section, we give some preliminary results and definitions that will be used in the sequel.
Throughout this paper, unless stated otherwise, let $K$ be a nonempty closed and convex subset of a real Hilbert space $H$. We first recall some well known concepts on monotonicity and continuity of real bifunctions. % that will be used in the sequel.

\begin{defi}
A bifunction $f:K\times K \rightarrow \mathbb{R}$ is called: 
\begin{itemize} 
\item [$(i)$] monotone if $f(x,y)+f(y,x)\leq0$ for all $x,y \in K$;
\item [$(ii)$] $\gamma$-strongly monotone, if there exists $\gamma >0$ such that $$f(x,y)+f(y,x)\leq -\gamma \|x-y\|^2 \; \text{for all}\;x,y \in K;$$
\item [$(iii)$] upper hemicontinuous, if $$\lim_{t\searrow0}
f(tz+(1-t)x,y)\leq f(x,y)\;\text{ for all} \;x,y,z\in K;$$
\item [$(iv)$] lower semicontinuous at $y$ with respect to the second argument on $K$, if $$f(x,y)\leq \displaystyle \liminf_{w\rightarrow y}f(x,w)\;\text{ for all}\; x\in K;$$
\item [$(v)$] an equilibrium bifunction, if for each $x\in K$, $f(x,x)=0$ and $ f(x,\cdot)$ is convex and lower semicontinuous.
\end{itemize}	    
\end{defi}

%%%%%%%%%%%%%%%%%%%%%%%%%%%%%%%%%%%
%One of the frequently used methods in the recent literature obtain solutions to $(EP)$ via Minty solutions to the following dual equilibrium problem: find $\bar{x}\in K$ such that
 The dual equilibrium problem associated with the bifunction $f$ on $K$ is stated as follows: 
find $\bar{x}\in K$ such that
\begin{equation*}\label{dep}
\tag{$DEP$}
f(y,\bar{x})\leq 0, \;\; \forall y\in K.
\end{equation*}
The set of solutions to $(DEP)$ is called the \textit{Minty solution set}. 
The following result gives the link between Minty equilibria and  the standard ones.
%%%%%%%%%%%%%%%%%%%%%%%%%%%%%%%%%%%%%%%%%%

\begin{lem}[Minty's Lemma, \cite{BO}]\label{lem-Minty}
	\begin{itemize}
		\item[(i)] Whenever $f$ is monotone, every
		solution of $(EP)$ is a solution of $(DEP)$.
		\item[(ii)] Conversely, if $f$ is upper hemicontinuous and equilibrium bifunction, then each solution of $(DEP)$ is a solution of $(EP)$. 
	\end{itemize}
\end{lem}
%%%%%%%%%%%%%%%%%%%%%%%%%%%%%%%%%%%%%
The next lemma introduces the notion of \textit{resolvent}  associated to bifunctions. This concept is crucial in our approach for solving $(BEP)$.

\begin{lem}\cite{CR2} \label{lem0}
	Suppose that $f:K\times K\rightarrow \mathbb{R}$ is a monotone equilibrium bifunction. Then the following are equivalent:
	%satisfies  $(H_{1}),(H_2),(H_{4})$. 
	\begin{itemize}
		\item [(i)] $f$ is maximal: $(x,u)\in K\times H$ and $f(x,y)\leq \langle u,x-y \rangle, \; \forall y\in K$ imply that $f(x,y)+\langle u,x-y \rangle \geq 0\; \forall y\in K$;
		\item [(ii)] for each $x\in H$ and $\lambda > 0$, there exists a unique  $z_{\lambda}=J^{f}_{\lambda}(x)\in K$,  called the resolvent of $f$ at $x$,  such that
		\begin{equation}\label{aa}
		\lambda f(z_{\lambda},y) +\langle y-z_{\lambda} ,z_{\lambda}-x  \rangle \geq 0, \;\; \forall y\in K.
		\end{equation}
	\end{itemize}
	Moreover, $\bar{x}\in S_f$ if, and only if, $\bar{x} = J^{f}_{\lambda}(\bar{x})$ for every $\lambda>0$ if, and only if, $\bar{x} = J^{f}_{\lambda}(\bar{x})$ for some $\lambda>0$.
\end{lem}

%%%%%%%%%%%%%%%%%%%%%%%%%%%%%%%%%%%%%%%%%%%%%%%%%%%
For the first main result of Section $3$ concerning the weak convergence of the sequence generated by algorithm \eqref{algo},
% that we formulate and prove afterwards,
we will make use of the two following useful lemmata.
\begin{lem}[discrete Opial Lemma, \cite{opial}]\label{disc-opial}
	Let $C$ be a nonempty subset of $H$ and $(x_k)_{k\geq0}$ be a
	sequence in $H$ such that the following two conditions hold:
	\begin{itemize}
		\item [(i)] For every $x\in C$, $\displaystyle \lim_{k\rightarrow +\infty}\|x_k-x\|$ exists.
		\item [(ii)] Every weak sequential cluster point of $(x_k)_{k\geq0}$ is in $C$.
	\end{itemize}
	Then, $(x_k)_{k\geq0}$ converges weakly to an element in $C$.
\end{lem}
%%%%%%%%%%%%%%%%%%%%%%%%%%%%%%%%%%%%%%%%%%%%%%%%%%%
\begin{lem}\label{lem-sum}
	Let $0\leq p\leq 1$, and let $\{b_k\}$ and $\{w_k\}$ be two sequences of nonnegative numbers such that, for all $k\geq 0$, 
	\begin{equation*}
		b_{k+1}\leq pb_k+w_k.
	\end{equation*} 
	If \;$\sum_{k=0}^{+\infty} w_k< +\infty$, then \;$\sum_{k=0}^{+\infty} b_k< +\infty$.
\end{lem}

\begin{proof}
	We have $$(1-p)b_k\leq b_k-b_{k+1}+w_k.$$	
	Summing up from  $k=0$ to $n$, we get
	\begin{align*}
		(1-p)\sum_{k=0}^{n}b_k&\leq \sum_{k=0}^{n}(b_k-b_{k+1})+\sum_{k=0}^{n}w_k\\
		&=b_0-b_{n+1}+\sum_{k=0}^{n}w_k \\
		&\leq b_0+\sum_{k=0}^{n}w_k.
	\end{align*}
	And since $1-p\geq 0$ and $\sum_{k=0}^{+\infty} w_k< +\infty$, we conclude that $\sum_{k=0}^{+\infty} b_k< +\infty$.
\end{proof}
We also need the following technical lemmata.
\begin{lem}\cite{Bauchk}\label{lem5-a}
	For all $x, y \in H$ and $\beta\in\mathbb{R},$ the following equality holds,
	\begin{equation*}
	\|\beta x+ (1-\beta)y\|^{2} = \beta\|x\|^{2} + (1-\beta)\|y\|^{2} -\beta(1-\beta)\|x-y\|^{2}.
	\end{equation*}
\end{lem}

\begin{lem}\cite{CR1}\label{lem3a}
	Let $\{a_n\}$ be a sequence of real numbers that does not decrease at
	infinity, in the sense that there exists a subsequence $\{a_{n_k}\}_ {k\geq0}$ of $\{a_n \}$ which satisfies
	$$a_{n_k} < a_{n_k+1} \quad \mbox{for all} \quad k \geq 0.$$ Then, the sequence of integers $\{\sigma(n)\}_{n\geq n_0}$ defined
	by
	$\sigma(n) := \max\{k \leq n \;:\; a_k < a_{k+1} \}$
	is a nondecreasing sequence verifying $\displaystyle \lim_{n\rightarrow +\infty} \sigma(n) = \infty$ and, for all $n \geq n_0$
	$$a_{\sigma(n)} < a_{\sigma(n)+1}\quad \mbox{and}\quad a_{n}\leq a_{\sigma(n)+1} .$$
\end{lem}

In the rest of this section we recall some background material from convex analysis.
For a function $\varphi:H\rightarrow \mathbb{R}\cup{\lbrace+\infty\rbrace}$ we denote by $\dom \varphi=\{x\in H: \varphi(x)<+\infty\}$ its effective domain and
say that $\varphi$ is proper, if $\dom \varphi\neq \emptyset$. We also denote by $\min \varphi:= \displaystyle \inf_{x\in H} \varphi(x)$ the optimal objective value of the function $\varphi$ and by $\argmin \varphi:=\{x\in H:\varphi(x)=\min \varphi\}$ its set of global minima.

For a proper lower semicontinuous convex function $\varphi:H\rightarrow \mathbb{R}\cup{\lbrace+\infty\rbrace}$ and $x\in H$, let  $\varphi^*:H\rightarrow \mathbb{R}\cup{\lbrace+\infty\rbrace}$ be its {Fenchel conjugate  defined by  $\varphi^*(x):=\displaystyle \sup_{y\in H}\lbrace \langle x,y \rangle-\varphi(y) \rbrace$. 
If $\varphi=\delta_K$ is 
the indicator function of $K\subset H$, i.e., $\delta_K(x)=0$ if $x\in K$ and $+\infty$ otherwise, its Fenchel conjugate at $x^*\in H$ is the support function of $K$ at $x^*$, i.e.,  $\delta_K^*(x^*)=\sigma_K(x^*)=\displaystyle \sup_{y\in K}\langle x^*,y \rangle.$ The subdifferential of $\varphi$ at $x\in H$, with $\varphi(x)\in \R$ is the set $\partial \varphi(x):=\{v\in H: \varphi(y)\geq \varphi(x)+\langle v,y-x\rangle,\;\forall y\in H\}$. We take by convention $\partial \varphi(x):=\emptyset$ if $\varphi(x)=+\infty$. 
  
%%%%%%%%%%%%%%%%%%%%%%%%%%%%%%%%%%%%%%%%%%%%%%%%%%%
The normal cone to $K\subset H$ at $x\in H$ is
$$N_K(x)=\left\{\begin{array}{ll}
\{x^*\in H: \langle x^*,u-x\rangle\leq0,\;\forall u\in K\}&  \text{ if }  x\in K\\
\emptyset& \text{otherwise}.
\end{array}\right.$$
We mention that $N_K=\partial\delta_K$, and that $x^*\in N_K(x)$ if, and only if, $\sigma_K(x^*)=\langle x^*,x\rangle$.
For every $u\in K$, we denote by $f_u$ the function defined  on $H$ by $f_u(x)=f(u,x)$ if $x\in K$ and $f_u(x)=+\infty$ otherwise.
%%%%%%%%%%%%%%%%%%%%%%%%%%%%%%%%%%%%%%%%%%
For an equilibrium bifunction $f:K\times K \rightarrow \mathbb{R}$, the associate operator $A^f$ is defined by $$A^f(x):= \partial f_x(x)=\left\{\begin{array}{ll} \lbrace z\in H:f(x,y)+\langle z,x-y \rangle\geq0, \;\forall y\in K \rbrace&  \text{ if }  x\in K\\
\emptyset& \text{otherwise}.
\end{array}\right.$$ 
%%%%%%%%%%%%%%%%%%%%%%%%%%%%%%%%%%%%%%%%%%%%%%%%%%%
The Fitzpatrick transform $\mathcal{F}_f : K\times H\rightarrow \R\cup \lbrace +\infty \rbrace$  associated to a bifunction $f$ and introduced in \cite{AH,BG}, is defined by 
$$\mathcal{F}_f(x,u)=\displaystyle \sup_{y\in K}\lbrace \langle u, y\rangle+f(y, x)\rbrace.$$ Given its continuity and convexity properties, the function $\mathcal{F}_f $ has proven to be an important tool when studying the asymptotic properties of dynamical equilibrium systems,
see \cite{CMR} for a detailed presentation of these elements.
This  section  concludes with  the following auxiliary result    needed for establishing our results.

\begin{propo}\cite{CMR}\label{fitz-phi}
	If $f(x,y)=\varphi(y)-\varphi(x)$ where $\varphi :H\rightarrow \mathbb R\cup\{+\infty\}$ is convex and lower semicontinuous with $\dom \varphi\subset K$, then for every $(x,u)\in K\times H$, $\mathcal{F}_f(x, u)=\varphi(x)+\varphi^*(u).$ \end{propo}
%%%%%%%%%%%%%%%%%%%%%%%%%%

\section{The main results}
%%%%%%%%%%%%%%%%%%%%%%%%%%%%%%%%%%%%%%%%%%%%%%%%%%%
In the remaining part of the paper, $f$ and $g$ are two monotone and upper hemicontinuous  bifunctions. We suppose that for  each $y\in K, \partial g_y(y)\neq \emptyset$ (i.e., $\dom (A^{g})=K$) and that $K \cap S_f \neq \emptyset$ and $\R_+ (K-S_f)$ is a closed linear subspace of $H$. In this case, the operator $g_x+\delta_{S_f}$ is maximally monotone, see \cite{Att-Riahi-Thera, Riahi5}, and the subdifferential sum formula $\partial (g_x+\delta_{S_f})=\partial g_x+N_{S_f}$ holds.
%%%%%%%%%%%%%%%%%%%%%%%%%%%%%%%%%%%%%%%%%%%%%%%%%%%
The following geometric assumption will be also needed and considered as a key tool in our treatment of the convergence analysis: $\forall u\in S_f, \;\text{for all}\; p\in N_{S_f}(u)$, 
\begin{equation}\label{fitz-discret}
	%\forall u\in S_f, \forall p\in N_{S_f}(u),\;
	\sum_{n=1}^{+\infty} \lambda_n\beta_n\left[\mathcal{F}_f\left(u,\dfrac{2p}{\beta_n}\right)-\sigma_{S_f}\left(\dfrac{2p}{\beta_n}\right)\right]<+\infty. \end{equation}
Let us mention that hypothesis \eqref{fitz-discret} is the 
discrete counterpart of the condition introduced in \cite{CMR} in the context of continuous-time dynamical equilibrium systems. Note also that it is a natural extension of similar assumptions known in the literature for the convergence analysis of variational inequalities expressed as monotone inclusion problems and for constrained convex optimization problems, see \cite{attouch2011prox, BC1, BC2} and references therein for further useful comments on these assumptions.

\subsection{Weak convergence analysis}

%We first recall the following summability result for real sequences that will be needed in our main results of this section.
%We are now in the position to prove the main result of this section. In order to do so, 

In this paragraph, under natural conditions, we obtain weak convergence result for the trajectory generated by \eqref{algo} to a solution of $(BEP)$. 
We first prove the following preliminary estimation.
\begin{lem}\label{lem5}
Let $\lbrace x_n \rbrace$ be a sequence generated by algorithm \eqref{algo}. Take $u\in S$ and set $a_n:=\|x_n-u\|^2$. Then, there exists $p\in N_{S_f}(u)$ such that for each $n\geq 1$ the following inequality holds:
	\begin{equation}\label{estim1}
		\begin{array}{l}
a_{n+1}-a_n-\alpha(a_n-a_{n-1})+\lambda_n\beta_nf(u,x_{n+1})\\
			\leq(\alpha-1)\|x_{n+1}-x_{n}\|^2 +2\alpha\|x_n-x_{n-1}\|^2	+\lambda_n\beta_n\left[\mathcal{F}_f\left(u,\dfrac{2p}{\beta_n}\right)-\sigma_{S_f}\left(\dfrac{2p}{\beta_n}\right)\right].
				\end{array}
\end{equation}
\end{lem}

\begin{proof}
Since $\{ x_n\}$ is generated by algorithm \eqref{algo}, we have for each $x\in K$
\begin{equation}\label{15a}
\begin{array}{l}
0\leq  \lambda_n \beta_n f(x_{n +1} ,x)+\lambda_n g(x_{n +1} ,x) \\
\qquad+\dfrac{1}{2}\left(\|y_n-x\|^2-\|x_{n+1}-x\|^2 -\|x_{n+1}-y_n\|^2\right).
\end{array}
%\begin{array}{rcl}
%0\leq  \lambda_n \beta_n f(x_{n +1} ,x)+\lambda_n g(x_{n +1} ,x) \\+ \dfrac{1}{2}\left(\|y_n-x\|^2-\|x_{n+1}-x\|^2 -\|x_{n+1}-y_n\|^2\right).
%\end{array}
\end{equation}
By Lemma \ref{lem5-a}, we have for all $n\geq 1$   	\begin{align}\label{8a}
\|y_n-x\|^2&=\|x_n+\alpha(x_n-x_{n-1})-x\|^2 \nonumber\\
&=\|(1+\alpha)(x_n-x)-\alpha(x_{n-1}-x)\|^2\nonumber\\
&=(1+\alpha)\|x_n-x\|^2-\alpha \|x_{n-1}-x\|^2+\alpha(1+\alpha)\|x_n-x_{n-1}\|^2.
\end{align}
Also, we have 	
\begin{equation}\label{8b}
\begin{array}{lll}
\|x_{n+1}-y_n\|^2& = & \|x_{n+1}-x_n-\alpha(x_n-x_{n-1})\|^2 \\ 
& = & \|x_{n+1}-x_n\|^{2}+\alpha^2\|x_{n}-x_{n-1}\|^{2}-2\alpha\langle x_{n+1}-x_n,x_{n}-x_{n-1}\rangle\\
& \geq &(1-\alpha)\|x_{n+1}-x_n\|^{2}+(\alpha^2-\alpha)\|x_{n}-x_{n-1}\|^{2}.
\end{array} 
\end{equation}	
Combining \eqref{8a} and \eqref{8b} with \eqref{15a}, we get for every $x\in K$
\begin{equation}\label{15principal}
\begin{array}{l}
\|x_{n+1}-x\|^{2}-(1+\alpha)\|x_{n}-x\|^{2}+\alpha\|x_{n-1}-x\|^{2}\\ \leq  (\alpha-1)\|x_{n+1}-x_{n}\|^2+2\alpha\|x_n-x_{n-1}\|^2+2\lambda_n \beta_n f(x_{n +1},x)+2\lambda_n g(x_{n +1} ,x).
\end{array}
\end{equation}
Since $u\in S$, according to the first-order optimality condition, we have $$0\in \partial(g_u+\delta_{S_f})(u)=A^g(u)+N_{S_f}(u).$$ Let $p\in N_{S_f}(u)$ be such that $-p\in A^g(u)$, we have for every $n\geq 1$
\begin{equation}\label{g1}
\lambda_ng(u,x_{n+1})+\lambda_n\langle -p,u-x_{n+1}\rangle \geq 0,
\end{equation} 
and by taking $x=u$ and $a_n=\|x_n-u\|^{2}$ in \eqref{15principal}, we also have
\begin{equation}\label{15c}
\begin{array}{rcl}
a_{n+1}-(1+\alpha)a_n+\alpha a_{n-1}&\leq & (\alpha-1)\|x_{n+1}-x_{n}\|^2+2\alpha\|x_n-x_{n-1}\|^2\\
&&+2\lambda_n \beta_n f(x_{n +1},u)+2\lambda_n g(x_{n +1} ,u).
\end{array}
\end{equation}
By summing up the above inequalities and using the monotonicity of $g$, we get
\begin{align*}
a_{n+1}-a_n-\alpha(a_n-a_{n-1})&\leq(\alpha-1)\|x_{n+1}-x_{n}\|^2+2\alpha\|x_n-x_{n-1}\|^2\\
&\quad+2\lambda_n\beta_n f(x_{n+1},u)+2\lambda_n\langle -p,u-x_{n+1}\rangle.
\end{align*} 
Using the monotonicity of $f$, we obtain
\begin{equation*}
	\begin{array}{l}
		a_{n+1}-a_n-\alpha(a_n-a_{n-1})+\lambda_n\beta_n f(u,x_{n+1})\\\leq(\alpha-1)\|x_{n+1}-x_{n}\|^2+2\alpha\|x_n-x_{n-1}\|^2
		+\lambda_n\beta_n f(x_{n+1},u)+2\lambda_n\langle -p,u-x_{n+1}\rangle\\
		=(\alpha-1)\|x_{n+1}-x_{n}\|^2+2\alpha\|x_n-x_{n-1}\|^2\\\quad\quad+\lambda_n\beta_n\left[\left\langle \frac{2p}{\beta_n},x_{n+1}\right \rangle+f(x_{n+1},u)-\left\langle \frac{2p}{\beta_n},u\right \rangle \right].
	\end{array}
\end{equation*}
Finally, using the fact that $p\in N_{S_f}(u)$, i.e., $\delta_{S_f}(u)+\sigma_{S_f}(p)=\langle p,u\rangle$, we obtain
\begin{equation*}
	\begin{array}{lll}
a_{n+1}-a_n-\alpha(a_n-a_{n-1})+\lambda_n\beta_n f(u,x_{n+1})&&\\
		\leq(\alpha-1)\|x_{n+1}-x_{n}\|^2+2\alpha\|x_n-x_{n-1}\|^2&&\\
		\qquad+\lambda_n\beta_n\left[\sup_{x\in H} \left \lbrace \left\langle \frac{2p}{\beta_n},x\right \rangle+f(x,u)\right \rbrace-\sigma_{S_f}\left(\frac{2p}{\beta_n}\right) \right]&&\\
		=(\alpha-1)\|x_{n+1}-x_{n}\|^2+2\alpha\|x_n-x_{n-1}\|^2\\
		\qquad+\lambda_n\beta_n\left[\mathcal{F}_f\left(u,\dfrac{2p}{\beta_n}\right)-\sigma_{S_f}\left(\dfrac{2p}{\beta_n}\right)\right].
	\end{array}
\end{equation*}

The proof is complete.
\end{proof}

\begin{rem}
We can continue our analysis assuming that $\displaystyle \sum_{n=1}^{+\infty}\|x_n-x_{n-1}\|^2<+\infty$; however this condition involves the trajectory $\{x_n\}$ which is unknown. 
In the next corollary we prove that the above condition holds under a suitable control of the parameter $\alpha$.
%Note that a similar assumption was imposed in \cite{moud1} in the framework of solving equilibrium problem $(EP)$ by an approximate second order procedure.  
%The main advantage of our approach is that it provides convergence without any restrictive assumption on the trajectories. 
\end{rem}

\begin{coro}\label{coro-disc}
Under hypothesis \eqref{fitz-discret} and by assuming that $0\leq \alpha< \frac{1}{3}$, we have 
\begin{itemize}
	\item [(i)] $\displaystyle \sum_{n=1}^{+\infty}\|x_n-x_{n-1}\|^2<+\infty$;
	\item [(ii)] $\displaystyle\sum_{n=1}^{+\infty}\lambda_n\beta_nf(u,x_{n+1})<+\infty$,\; for each $u\in S$.
\end{itemize}
\end{coro}

\begin{proof}
(i)  First we simplify the writing of the estimation \eqref{estim1} given in Lemma \ref{lem5}. Since $u\in S_f$ and $\lambda_n\beta_n\geq 0$, we have $\lambda_n\beta_nf(u,x_{n+1})\geq 0$. Setting $\delta_n=\|x_n-x_{n-1}\|^2$, then inequality \eqref{estim1} gives
\begin{equation}\label{pluie}
a_{n+1}-a_n-\alpha(a_n-a_{n-1})+(1-\alpha)\delta_{n+1}-2\alpha\delta_n\leq\lambda_n\beta_n\left[\mathcal{F}_f\left(u,\dfrac{2p}{\beta_n}\right)-\sigma_{S_f}\left(\dfrac{2p}{\beta_n}\right)\right]
.
\end{equation}
 In order to simplify its summation  we rewrite  (\ref{pluie}) as
\begin{equation}\label{soleil}
\begin{array}{l}
a_{n+1}-a_n-\alpha(a_n-a_{n-1})+(1-\alpha)(\delta_{n+1}-\delta_n)+(1-3\alpha)\delta_n\\
\leq\lambda_n\beta_n\left[\mathcal{F}_f\left(u,\dfrac{2p}{\beta_n}\right)-\sigma_{S_f}\left(\dfrac{2p}{\beta_n}\right)\right]
.
\end{array}
\end{equation}

Now, summing up (\ref{soleil})  from $j=1$ to $n$, we obtain   
$$
\begin{array}{l}
(a_{n+1}-a_1)-\alpha(a_n-a_0)+(1-\alpha)(\delta_{n+1}-\delta_1)+(1-3\alpha)\sum_{j=1}^n\delta_j\\
\leq \sum_{j=1}^n\lambda_j\beta_j\left[\mathcal{F}_f\left(u,\dfrac{2p}{\beta_j}\right)-\sigma_{S_f}\left(\dfrac{2p}{\beta_j}\right)\right].
\end{array}
$$  
Assumption \eqref{fitz-discret}, infers that
\begin{equation}\label{C1}
(a_{n+1}-\alpha a_n)+(1-\alpha)\delta_{n+1}+(1-3\alpha)\sum_{j=1}^n\delta_j\leq C
,\end{equation}
for some nonnegative constant $C$.\\
Since $\alpha <\frac{1}{3}$ yields $1-3\alpha>0$ and $1-\alpha>0$, then inequality \eqref{C1} implies for all $n\geq 1$
\begin{equation}\label{C2}
a_{n+1}\leq \alpha a_n+C. \end{equation}
Recursively we obtain for all $n\geq n_0\geq 1$
\begin{align*}
a_{n+1}&\leq \alpha^{n-n_0}a_{n_0}+C(1+\alpha+\alpha^2+...+\alpha^{n-n_0-1})
\\
&=\alpha^{n-n_0}a_{n_0}+C\frac{1-\alpha^{n-n_0}}{1-\alpha}.
\end{align*}
%which implies that $\displaystyle \sup_na_n<+\infty$. 
Therefore the sequence $\{x_n\}$ is bounded and since 
\begin{equation}\label{C4}
	\sup_n\|x_{n+1}-x_n\|\leq 2\sup_n\|x_n\| <+\infty,
	\end{equation} the sequence $\{\delta_n\}$ is also bounded.
%\begin{equation}\label{C4}
%	\sup_n\|x_{n+1}-x_n\|\leq 2\sup_n\|x_n\| <+\infty,
%	\end{equation}
Combining \eqref{C4}  with \eqref{C1} and noticing that $1-3\alpha>0,$ yields  
$$\sum_{j=1}^{+\infty}\delta_j<+\infty,$$
ensuring $(i).$ \\
Returning to inequality \eqref{estim1}, we have
\begin{align*}
a_{n+1}-a_n-\alpha(a_n-a_{n-1})+\lambda_n\beta_nf(u,x_{n+1})&\leq \lambda_n\beta_n\left[\mathcal{F}_f\left(u,\dfrac{2p}{\beta_n}\right)-\sigma_{S_f}\left(\dfrac{2p}{\beta_n}\right)\right]\\&\qquad+\underbrace{(\alpha-1)}_{\leq 0}\delta_{n+1}+2\alpha\delta_n.\\
&\leq\lambda_n\beta_n\left[\mathcal{F}_f\left(u,\dfrac{2p}{\beta_n}\right)-\sigma_{S_f}\left(\dfrac{2p}{\beta_n}\right)\right]\\&\qquad+2\alpha\delta_n.\
\end{align*}
By summing up from $n=1$ to $+\infty$, we obtain
\begin{align*}
\sum_{n=1}^{+\infty} \lambda_n\beta_nf(u,x_{n+1})&\leq a_1-\alpha a_0+\sum_{n=1}^{+\infty} \lambda_n\beta_n\left[\mathcal{F}_f\left(u,\dfrac{2p}{\beta_n}\right)-\sigma_{S_f}\left(\dfrac{2p}{\beta_n}\right)\right]\\&\qquad+2\alpha\sum_{n=1}^{+\infty}\|x_n-x_{n-1}\|^2.
\end{align*}
Then, assumptions \eqref{fitz-discret} and $(i)$ ensure $(ii)$. 
\end{proof}
%In order to do so, we first recall two useful lemmas.

In order to further proceed with the convergence analysis, we have to choose the sequences $\{\lambda_n\}$ and $\{\beta_n\}$ such that $\displaystyle \liminf_{n\rightarrow +\infty}\lambda_n>0$ and $\beta_n\rightarrow +\infty$. 
We are now able to state and prove the first main result of this section. 

\begin{theo}\label{thm-weak}
Suppose given monotone and upper hemicontinuous bifunctions $f$ and $g$. Let $\lbrace x_n \rbrace$ be a sequence generated by algorithm \eqref{algo}.
Under hypothesis \eqref{fitz-discret} and by assuming that 
\begin{center}
$0\leq \alpha< \frac{1}{3}$, \; $\displaystyle \liminf_{n\rightarrow +\infty}\lambda_n>0$ \; and \; $\displaystyle \lim_{n\rightarrow +\infty}\beta_{n}=+\infty$,	
\end{center} 
the sequence $\lbrace x_n \rbrace$ weakly converges to $\bar{x}\in S$.
\end{theo}

\begin{proof}
The proof relies on the discrete Opial Lemma. To this end we will prove
that the conditions (i) and (ii) in Lemma \ref{disc-opial} for $C=S$ are satisfied.\\
Returning to inequality \eqref{estim1}, since $u\in S_f$ and $\lambda_n\beta_n\geq 0$, we have $\lambda_n\beta_nf(u,x_{n+1})\geq 0$, and then 
\begin{equation*}
a_{n+1}-a_n\leq\alpha(a_n-a_{n-1}) +2\alpha\|x_n-x_{n-1}\|^2+\lambda_n\beta_n\left[\mathcal{F}_f\left(u,\dfrac{2p}{\beta_n}\right)-\sigma_{S_f}\left(\dfrac{2p}{\beta_n}\right)\right]
.
\end{equation*}
Taking the positive part, we immediately deduce that
\begin{equation*}
[a_{n+1}-a_n]_+\leq\alpha[a_n-a_{n-1}]_+ +2\alpha\|x_n-x_{n-1}\|^2 +\lambda_n\beta_n\left[\mathcal{F}_f\left(u,\dfrac{2p}{\beta_n}\right)-\sigma_{S_f}\left(\dfrac{2p}{\beta_n}\right)\right].
\end{equation*}
Using assumption \eqref{fitz-discret} together with the fact that $\displaystyle \sum_{n=1}^{+\infty}\|x_n-x_{n-1}\|^2<+\infty$ and applying Lemma \ref{lem-sum} with $$b_n=[a_n-a_{n-1}]_+ \;\text{and}\;w_n=2\alpha\|x_n-x_{n-1}\|^2 +\lambda_n\beta_n\left[\mathcal{F}_f\left(u,\dfrac{2p}{\beta_n}\right)-\sigma_{S_f}\left(\dfrac{2p}{\beta_n}\right)\right],$$ we obtain    
$$\sum_{n=1}^{+\infty}[a_n-a_{n-1}]_+<+\infty.$$ 
Since $a_n$ is nonnegative, this implies the existence of $\displaystyle \lim_{n\rightarrow +\infty} a_n$ and the one of \\$\displaystyle \lim_{n\rightarrow +\infty}\|x_n-u\|$.

It remains to show that every weak cluster point $\bar{x}$ of the sequence $\lbrace x_n \rbrace$ lies in $S$.  
Let $n_k\rightarrow +\infty$ as $k\rightarrow +\infty$ such that $x_{n_k}\rightharpoonup \bar{x}$. We want to show that $\bar{x}\in S$.
Thanks to the monotonicity of $f$ and $g$, inequality \eqref{algo} ensures that for all $y\in K$ and for all $k$ large enough 
\begin{equation}\label{13}
f(y,x_{{n_k}+1})\leq -\frac{1}{\beta_{n_k}}g(y,x_{{n_k}+1})+\frac{1}{\lambda_{n_k}\beta_{n_k}}\langle x_{{n_k}+1}-y_{n_k}, y-x_{{n_k}+1}\rangle.
\end{equation}
%Using that $\beta_n\rightarrow +\infty$, $g(y,.)$ $\displaystyle \liminf_{n\rightarrow +\infty}\lambda_n\beta_n>0$, $\|x_{n+1}-y_n\|\rightarrow 0,$
Since $\partial g_y(y)\neq \emptyset$, one can find $x^*(y)\in H$ such that for every $z\in K$
$$
g(y,z)\geq \langle x^*(y),z-y\rangle \geq -\Vert x^*(y) \Vert\cdot\Vert y-z \Vert.
$$ 
Thus there exists $\gamma(y):=\Vert x^*(y) \Vert>0$ such that for every $z\in K$
\begin{equation}\label{14}
-g(y,z)\leq \gamma(y).\Vert y-z \Vert.
\end{equation}
Returning to \eqref{13}, we can write 
\begin{equation*}
f(y,x_{n_k+1})\leq \dfrac{\gamma(y)}{\beta_{n_k}}\Vert y-x_{n_k+1} \Vert+\dfrac{1}{\lambda_{n_k}\beta_{n_k}}\Vert x_{n_k+1}-y_{n_k} \Vert.\Vert y-x_{n_k+1} \Vert.
\end{equation*}
Passing to the limit, and using the facts that $\{x_{n_k}\}$ is bounded, $\{\beta_{n_k}\}\rightarrow +\infty$, $\displaystyle \liminf_{k\rightarrow +\infty}\lambda_{n_k}>0$ and $\|x_{n_k+1}-y_{n_k}\|\rightarrow 0$, we deduce that $f(y,\bar{x})\leq 0$ for all $y\in K$. Lemma \ref{lem-Minty} leads to $\bar{x}\in S_f$.

By using \eqref{algo} and the monotonicity of $f$ and $g$, we have for every $u\in S_f$,
\begin{equation*}
\lambda_n\beta_n f(u,x_{n+1})+\lambda_ng(u,x_{n+1})\leq \langle y_n-x_{n+1}, x_{n+1}-u\rangle.
\end{equation*}
By exploiting that $\displaystyle \lim_{n\rightarrow +\infty}\|x_n-u\|$ exists and thanks to $(i)$ of Corollary \ref{coro-disc}, we deduce that 
\begin{equation*}
\displaystyle \langle y_n-x_{n+1}, x_{n+1}-u\rangle \rightarrow_{n\rightarrow +\infty} 0.
\end{equation*}
Using $(ii)$ of the same Corollary, we obtain that $\displaystyle \limsup_{n\rightarrow +\infty}\lambda_ng(u,x_{n+1})\leq 0$. Since $g(u,.)$ is lower semicontinuous, from the assumption $\displaystyle \liminf_{n\rightarrow +\infty}\lambda_n>0$ we derive that $g(u,\bar{x})\leq 0$. Lemma \ref{lem-Minty} allows us to conclude that $$g(\bar{x},u)\geq 0,\;\forall u\in S_f,$$
%i.e., $\bar{x}\in S$. 
establishing the proof. \end{proof}

\subsection{Strong convergence analysis}

%Here, we suppose that the objective bifunction $g$ is $\rho$-strongly monotone, 
In this paragraph, under additional assumption on the monotonicity of the bifunction of the upper level $g$, we ensure the strong convergence of the trajectory in \eqref{algo}. 

\subsubsection{Strong convergence under assumption \eqref{fitz-discret}}

\begin{theo}\label{strong1}
Suppose that the bifunctions $f$ and $g$ are monotone and upper hemicontinuous. Under hypothesis \eqref{fitz-discret}, if the bifunction $g$ is $\rho$-strongly monotone, and if $$0\leq \alpha< \frac{1}{3} \;\text{ and } \;\sum_{n=1}^{\infty}\lambda_n=+\infty, $$ the sequence $\{x_n\}$ generated by algorithm \eqref{algo} strongly converges to a unique solution $u\in S$.
\end{theo}

\begin{proof}
Uniqueness of the solution for $(BEP)$ follows from strong
monotonicity of $g$. For the existence, see \cite[Theorem 4.3]{CCR}.\\
Using inequalities \eqref{g1} and \eqref{15principal}, with $x=u$, by summing up and using  the $\rho$-strong monotonicity of $g$, we get  for $a_n:=\|x_n-u\|^2$
\begin{align*}
a_{n+1}-a_n-\alpha(a_n-a_{n-1})&\leq-2\rho\lambda_n a_{n+1}+(\alpha-1)\|x_{n+1}-x_{n}\|^2+2\alpha\|x_n-x_{n-1}\|^2\\
&\quad+2\lambda_n\beta_n f(x_{n+1},u)+2\lambda_n\langle -p,u-x_{n+1}\rangle.
\end{align*} 
We follow the arguments in the proof of Lemma \ref{lem5} to obtain
\begin{align*}
a_{n+1}-a_n-\alpha(a_n-a_{n-1})+2\rho\lambda_n a_{n+1}&\leq(\alpha-1)\|x_{n+1}-x_{n}\|^2+2\alpha\|x_n-x_{n-1}\|^2 \nonumber\\ 
&\quad+\lambda_n\beta_n\left[\mathcal{F}_f\left(u,\dfrac{2p}{\beta_n}\right)-\sigma_{S_f}\left(\dfrac{2p}{\beta_n}\right)\right].
\end{align*}

Then, by summing up from $n=1$ to $+\infty$, we obtain
\begin{align*}
2\rho\sum_{n=1}^{+\infty} \lambda_n\|x_{n+1}-u\|^2&\leq a_1-\alpha a_0+\sum_{n=1}^{+\infty} \lambda_n\beta_n\left[\mathcal{F}_f\left(u,\dfrac{2p}{\beta_n}\right)-\sigma_{S_f}\left(\dfrac{2p}{\beta_n}\right)\right]\\
&\quad+2\alpha\sum_{n=1}^{+\infty}\|x_n-x_{n-1}\|^2.
\end{align*}  
Using condition \eqref{fitz-discret} and assumption $(i)$ of Corollary \ref{coro-disc}, we deduce that 
\begin{equation*}
\sum_{n=1}^{+\infty} \lambda_n\|x_{n+1}-u\|^2<+\infty.
\end{equation*}
Since $\displaystyle \lim_{n\rightarrow +\infty}\|x_n-u\|$ exists and $\sum_{n=1}^{\infty}\lambda_n=+\infty$, we conclude that $\displaystyle \lim_{n\rightarrow +\infty}\|x_n-u\|=0$, which guarantees the strong convergence of the whole sequence $\{x_n\}$ to $u$.
\end{proof}

%%%%%%%%%%%%%%%%%%%%%%%%%%%%%%%%%%%%%%%%%%%%%%%%%%%%
\subsubsection{Strong convergence without assumption \eqref{fitz-discret}}

We will show that in this case, the algorithm strongly
converges without the need of the geometric hypothesis \eqref{fitz-discret}.

\begin{theo}\label{strong2}
	Suppose that the bifunctions $f$ and $g$ are monotone and upper hemicontinuous with $S_f\neq\emptyset$ and $g$ is $\rho$-strongly monotone. Suppose moreover that \begin{center}
	$0\leq \alpha< \frac{1}{3}$, $\displaystyle \lim_{n\rightarrow +\infty}\lambda_n =0$, $\displaystyle\sum_{n=0}^{+\infty} \lambda_n=+\infty$, $\displaystyle \lim_{n\rightarrow +\infty}\beta_{n}=+\infty$ \; and \; $\displaystyle \liminf_{n\rightarrow +\infty}\lambda_n\beta_n>0.$ 	
	\end{center}
	Then, the sequence $\{ x_n\}$ generated by algorithm \eqref{algo} converges strongly to the unique solution $u$ of $(BEP)$.
\end{theo}

\begin{proof}
	Under assumptions on the two bifunctions $f$ and $g,$ we get the unique solution denoted by $\bar{x}$ of the
	bilevel equilibrium problem $(BEP)$. \\
	
	\noindent {\textbf{Step 1: We show that $\{ x_n\}$ is bounded.}}\\
	Since $\{ x_n\}$ is generated by algorithm \eqref{algo}, then by \eqref{15principal}, we have for each $x\in K$
	\begin{equation}\label{15}
	\begin{array}{l}
	\|x_{n+1}-x\|^{2}-(1+\alpha)\|x_{n}-x\|^{2}+\alpha\|x_{n-1}-x\|^{2}\\\leq  (\alpha-1)\|x_{n+1}-x_{n}\|^2+2\alpha\|x_n-x_{n-1}\|^2+2\lambda_n \beta_n f(x_{n +1},x)+2\lambda_n g(x_{n +1} ,x).
	\end{array}
	\end{equation}
	Fix $x\in S_f$, and set $a_{n }(x)= \|x_n-x\|^2$ and $\delta_{n }=\|x_n-x_{n-1}\|^2.$  Thanks to the monotonicity of $f$, then for each $n\geq 0$,
	\begin{equation}\label{20}
		\begin{array}{l}
	a_{n+1}(x)-\alpha a_n(x)+2\alpha\delta_{n+1}\\\leq \left(a_n(x)-\alpha a_{n-1}(x)+2\alpha\delta_n\right)+(3\alpha-1)\delta_{n+1} +2\lambda_ng(x_{n+1},x).
	\end{array} \end{equation} 
	By setting $b_n(x)=a_n(x)-\alpha a_{n-1}(x)+2\alpha\delta_n,$ we obtain, for $n\geq 1,$
	\begin{equation}\label{keyac}
	b_{n+1}(x)\leq b_n(x)+(3\alpha-1)\delta_{n+1} +2\lambda_ng(x_{n+1},x).
	\end{equation}
	\begin{itemize}
		\item 
		If there is $n_0 \in \mathbb{N}$ such that $\{b_{n} (x)\}$ is decreasing  for all $n \geq n_0 ,$ then $b_{n} (x)\leq b_{n_0} (x),$ which infers that
		$$a_{n+1}(x)\leq \alpha a_n(x)+b_{n_0}\quad \mbox{for all}\;\; n \geq n_0.$$
		Recursively, we obtain for all $n\geq n_0\geq 1$
		$$a_{n+1}(x)\leq \alpha^{n-n_0} a_{n_0}(x)+b_{n_0}\frac{1-\alpha^{n-n_0}}{1-\alpha},$$
		and the boundedness of the sequence $\{a_n (x)\}$.
		\item  Otherwise there exists an increasing sequence $\{k_n \}$ such that for every $n \geq 0,$ $b_{k_{n +1}} (x)> b_{k _n }(x).$ By Lemma \ref{lem3a}, there exist  a nondecreasing sequence $\{\sigma_n \}$
		and $n_0 > 0$ such that $\displaystyle \lim_{n\rightarrow +\infty} \sigma_n = \infty,$ and for all $n \geq n_0,$  $b_{\sigma_n(x)} < b_{\sigma_{n}+1}(x)$ and $b_n (x) \leq b_{\sigma_{n}+1}(x).$ For $n=\sigma_n$ in \eqref{keyac}, we get
		\begin{equation}
		0<b_{\sigma_{n}+1}(x)- b_{\sigma_{n}}(x)\leq (3\alpha-1)\delta_{\sigma_{n}+1} +2\lambda_{\sigma_{n}}g(x_{\sigma_{n}+1},x).
		\end{equation}
		Using the $\rho$-strong monotonicity of $g$ and relation \eqref{14},  we
		deduce that for $n\geq n_0$
		\begin{equation}
		-2\lambda_{\sigma_{n}}\gamma(x)\sqrt{a_{\sigma_{n}+1}(x)}\leq 2\lambda_{\sigma_{n}}g(x,x_{\sigma_{n}+1})\leq (3\alpha-1)\delta_{\sigma_{n}+1} -2\lambda_{\sigma_{n}}\rho a_{\sigma_{n}+1}(x).
		\end{equation}
		Since $3\alpha-1< 0,$ we conclude that  for $n\geq n_0$
		\begin{equation}
		a_{\sigma_{n}+1}(x)\leq\left( \dfrac{\gamma(x)}{\rho}\right)^{2}\quad \mbox{and } \quad \delta_{\sigma_{n}+1}\leq \dfrac{2\gamma^{2}(x)\lambda_{\sigma_n}}{\rho(1-3\alpha)}.
		\end{equation}
		Hence, $\{a_{\sigma_{n}+1}(x)\}$ is bounded and  since, $\{\lambda_{\sigma_{n}}\}$ is bounded, then  $\{\delta_{\sigma_{n}+1}\}$ is  bounded, which means that $\{b_{\sigma_{n}}(x)\}$  also is  bounded. 
		So, for all $n \geq n_0,$ we have 
		\begin{equation*}
		\begin{array}{lll}
		a_n(x)& \leq & \alpha a_{n-1}+ b_n(x) \\ 
		& \leq  & \alpha a_{n-1}+b_{\sigma_{n}}(x)\\
		&\leq&\alpha a_{n-1}+C \\
		&\leq& \alpha^{n-n_0} a_{n_0}(x)+C\frac{1-\alpha^{n-n_0}}{1-\alpha}.
		\end{array} 
		\end{equation*}
		Therefore the sequence $\{ a_n(x)\}$ is bounded, ensuring the boundedness of $\{x_{n}\}.$
	\end{itemize}
	{\bf Step 2: We show that the sequence $\{ x_n\}$ strongly converges to $\bar{x}$, the unique solution of $(BEP).$}\\ Let us consider two cases:\\
	\underline{Case 1:} There exists $n_0$ such that $\{b_n(\bar{x})\}:=a_n(\bar{x})-\alpha a_{n-1}(\bar{x})+2\alpha\delta_n$ is decreasing for $n \geq n_0$. \\
	Then, the limit of sequence $\{b_n(\bar{x})\}$ exists and $\displaystyle \lim_{n\rightarrow +\infty}(b_n (\bar{x}) - b_{n+1}(\bar{x})) = 0.$ 
	Since $\bar{x}\in S_f,$ then by \eqref{keyac} we have
	\begin{equation}\label{20a}
	b_{n+1}( \bar{x})\leq b_n( \bar{x})+(3\alpha-1)\delta_{n+1} +2\lambda_ng(x_{n+1}, \bar{x}).
	\end{equation} 
	Hence, since $3\alpha-1<0,$ $\displaystyle \lim_{n\rightarrow +\infty}\lambda_n =0$ and $g(\cdot, \bar{x})$ is lower semi-continuous, then
	$\displaystyle \lim_{n\rightarrow +\infty}\delta_{n+1}=0.$\\
	Summing up 
	inequality \eqref{20a} from $1$ to $+\infty,$  we deduce that
	\begin{equation}
		 \sum_{n=1}^{+\infty} -\lambda_ng(x_{n+1},\bar{x})\leq b_1(\bar{x}),
	\end{equation}
	%which, implies that $\displaystyle \liminf_{n\rightarrow \infty}-g(x_{n+1},\bar{x})\leq0$ (because  $\displaystyle \sum_{n=0}^{+\infty} \lambda_n =+\infty$).\\
	which in combination with $\displaystyle \sum_{n=0}^{+\infty} \lambda_n =+\infty$ leads to $\displaystyle \liminf_{n\rightarrow \infty}-g(x_{n+1},\bar{x})\leq0$. \\
	On the other hand, since $g$ is $\rho$-strongly monotone, then we have
	\begin{equation}\label{k}
	\begin{array}{lll}
	\displaystyle \lim_{n\rightarrow +\infty} a_{n+1}(\bar{x})& = & \displaystyle \liminf_{n\rightarrow +\infty}\|x_{n+1}-\bar{x}\|^{2} \\ 
	& \leq & \frac{1}{\rho}\underbrace{\liminf_{n\rightarrow +\infty}-g(x_{n+1},\bar{x})}_{\leq 0}+\frac{1}{\rho}\displaystyle \limsup_{n\rightarrow +\infty}-g(\bar{x},x_{n+1})\\
	&\leq& -\frac{1}{\rho}\displaystyle \liminf_{n\rightarrow +\infty}g(\bar{x},x_{n+1}).
	\end{array} 
	\end{equation}
	Hence, to prove that  the sequence $\{{a_{n+1} (\bar{x})}\}$ converges to zero,  it is enough to verify
	that $\displaystyle \liminf_{n\rightarrow +\infty}g(\bar{x},x_{n+1}) \geq0.$ Since  $\{x_n\}$ is bounded, let $x$ be a weak cluster point of $\{x_n\},$ i.e.
	$x = \displaystyle w - \lim_{n \in I_n\subset \mathbb{N}} x_n.$ By using the weak lower semicontinuity of $g (\bar{x}, \cdot)$ we have
	$$
	g(\bar{x}, x) \leq  \liminf_{n\in I} g (\bar{x}, x_{n+1} ).$$
	Since $\bar{x}$ is the unique solution of $(BEP)$, we need just to check that $ x\in S_f.$  In doing so, by \eqref{14} and \eqref{15principal},  we have for every $y\in K,$ 
	\begin{equation}\label{co}
	f(y,x_{n+1})\leq -\frac{1}{\lambda_n\beta_{n}}\left(b_{n+1}(y)-b_n(y)\right)+ \frac{1}{2\beta_{n}}\gamma(y)\sqrt{a_{n}(x)}. \end{equation}
	We have 
	$$\begin{array}{l}
	b_{n}(y)-b_{n+1}(y)\\= \left(a_n(y)-\alpha a_{n-1}(y)+2\alpha\delta_n\right)-\left(a_{n+1}(y)-\alpha a_{n}(y)+2\alpha\delta_{n+1}\right)  \\ 
	= \left(a_n(y)- a_{n+1}(y)\right)+\alpha \left(a_n(y)- a_{n-1}(y)\right)+2\alpha\left(\delta_n-\delta_{n+1}\right)\\
	=\left(a_{n}(\bar{x})-a_{n+1}(\bar{x})+ 2\langle x_n-x_{n+1}, \bar{x}-y\rangle\right)\\\quad+\alpha\left(a_{n-1}(\bar{x})-a_{n}(\bar{x})+ 2\langle x_n-x_{n-1}, \bar{x}-y\rangle\right)
	+2\alpha\left(\delta_n-\delta_{n+1}\right)\\
	= b_{n}(\bar{x})-b_{n+1}(\bar{x})+2\langle x_n-x_{n+1}, \bar{x}-y\rangle+2\alpha \langle x_n-x_{n-1}, \bar{x}-y\rangle.
	\end{array}$$
	Since $\displaystyle \lim_{n\rightarrow +\infty}(b_n (\bar{x}) - b_{n+1}(\bar{x})) = 0$ and $\displaystyle \lim_{n\rightarrow +\infty}\|x_{n+1}-x_n\|= 0,$ then $\displaystyle \lim_{n\rightarrow +\infty} (b_{n}(y)-b_{n+1}(y))=0.$ \\
	By using the weak lower semicontinuity of $f (y, \cdot)$ and the fact that $\{x_n\}$ is bounded, $\displaystyle \lim_{n\rightarrow +\infty}\lambda_n=0,$ $\displaystyle \liminf_{n\rightarrow +\infty}\lambda_n\beta_n>0$ and $\displaystyle \lim_{n\rightarrow +\infty}\beta_{n}=+\infty,$ we conclude  from \eqref{co} that for every $y\in K$
	$$f (y, x) \leq  \liminf_{n\in I} f(y, x_{n+1} ) \leq 0.$$
	Hence, by using Minty’s lemma we deduce that $x\in S_f.$ \\
	Therefore, $$0 \leq g(\bar{x}, x) \leq  \liminf_{n\in I} g(\bar{x}, x_{n+1} ),$$  and so $\displaystyle \lim_{n\rightarrow +\infty}a_n(\bar{x})=0.$\\
	\underline{Case 2:} There exists a subsequence $\{x_{n_{j}}\}$ of $\{x_{n}\}$ such that $b_{n_{j}}(\bar{x})\leq b_{n_{j}+1}(\bar{x})$ for
	all $j \in \mathbb{N}.$ \\
	By Lemma \ref{lem3a}, 
	the sequence $\sigma(n):= \max\{k \leq n \;:\; b_k < b_{k+1} \}$
	is a nondecreasing, $\displaystyle \lim_{n\rightarrow +\infty} \sigma(n) = \infty$ and, for all $n\geq n_0$
	$$b_{\sigma(n)} < b_{\sigma(n)+1}\quad \mbox{and}\quad b_{n}\leq b_{\sigma(n)+1} .$$
	Let us take $n=\sigma(n)$ and $x=\bar{x}$ in \eqref{keyac}. We have
	\begin{equation}\label{keyad}
	0<b_{\sigma(n)+1}(\bar{x})-b_{\sigma(n)}(\bar{x})\leq  2\lambda_{\sigma(n)}g(x_{\sigma(n)+1},\bar{x}),
	\end{equation} 
	which yields $g(x_{\sigma(n)+1},\bar{x}) \geq 0,$ and thus $\displaystyle \limsup_{n\rightarrow +\infty} g(x_{\sigma(n)+1},\bar{x}) \leq 0.$\\
	Using again the $\rho$-strong monotonicity of $g$ and passing to the limit we have
	\begin{equation}\label{kb}
	\begin{array}{lll}
	\displaystyle \limsup_{n\rightarrow +\infty} a_{\sigma(n)+1}(\bar{x}) 
	& \leq & \frac{1}{\rho}\underbrace{\limsup_{n\rightarrow +\infty}-g(x_{\sigma(n)+1},\bar{x})}_{\leq 0}+\frac{1}{\rho}\displaystyle \limsup_{n\rightarrow +\infty}-g(\bar{x},x_{\sigma(n)+1})\\
	&\leq& -\frac{1}{\rho}\displaystyle \liminf_{n\rightarrow +\infty}g(\bar{x},x_{\sigma(n)+1}).
	\end{array} 
	\end{equation}
	Under the boundedness of $\{x_n\},$ and similarly to the case 1, one can show that 
	$$\displaystyle \liminf_{n\rightarrow +\infty}g(\bar{x}, x_{\sigma(n)+1} ) \geq 0.$$ Hence, by \eqref{k}, we conclude that $\displaystyle \lim_{n\rightarrow +\infty}a_{\sigma(n)+1}(\bar{x})=0.$
	
	Since $b_n(\bar{x}) \leq b_{\sigma(n)+1} (\bar{x})$ for each
	$n \geq n_0 ,$ we derive that
	$$ 
	\begin{array}{lll}
	\displaystyle \lim_{n\rightarrow +\infty} a_n(\bar{x})\leq  \displaystyle \lim_{n\rightarrow +\infty} b_n(\bar{x})&\leq&\displaystyle \lim_{n\rightarrow +\infty}b_{\sigma(n)+1} (\bar{x})\\
	& \leq & \displaystyle\lim_{n\rightarrow +\infty} \left(a_{\sigma(n)}(\bar{x})+2\alpha\delta_{\sigma(n)}\right) \\ 
	& \leq &(1+4\alpha) \displaystyle\lim_{n\rightarrow +\infty} a_{\sigma(n)}(\bar{x})+4\alpha \displaystyle \lim_{n\rightarrow +\infty} a_{\sigma(n)-1}(\bar{x}) \\ 
	& = & 0,
	\end{array}  $$
	thus guaranteeing the  strong convergence of the whole
	sequence $\{x_n \}$ to $\bar{x}.$	
\end{proof}

%%%%%%%%%%%%%%%%%%%%%%%%%%

\section{Application to optimization and saddle point problems}

% % % % % % % % % % % % % % % % % %
In this section, we give two examples of particular bifunctions, for which our main weak and strong convergence theorems apply.
\subsection{Hierarchical minimization}
Our contribution in this paragraph discusses the following hierarchical minimization problem:  
\begin{equation}\label{bil-prog}
\tag{$HMP$}
\min_{x\in  \underset{K}{\hbox{argmin}}\, \psi}\; \varphi(x),
\end{equation}
where $\psi: H\rightarrow \R\cup\{+\infty\}$ is a proper, convex and lower semicontinuous extend real-valued function such that $K=\dom \psi$ is closed and $\varphi:H \rightarrow \mathbb{R}\cup\{+\infty\}$ is a differentiable and lower semicontinuous function such that $K=\dom \varphi $ is closed.
The above problem can be equivalently expressed as:
\vskip 2mm \noindent 
find $\bar{x}\in \underset{K}{\hbox{argmin}}\,\psi$ such that
\begin{equation}\label{pbep}
%\text{find} \; \bar{x}\in \underset{K}{\hbox{argmin}}\, 
\varphi(\bar{x})\leq \varphi(y), \;\; \forall y\in \underset{K}{\hbox{argmin}}\,\psi.
\end{equation}

\noindent Clearly, \eqref{pbep} can be viewed as a bilevel equilibrium  problem $(BEP)$ such that the associated bifunctions are defined for all $x,y\in K$ by $f(x,y)=\psi(y)-\psi(x)$  and $g(x,y)=\varphi(y)-\varphi(x)$.
%We have  two  methods  to  represent  an optimization problem  $\min_K\psi$ by  (EP). The first one is to take $ f(x,y)= \psi (y)-\psi (x) $. The second one is to choose $ f (x, y) =  \psi^\prime  (x, y-x) $,  where $\psi^\prime (x; h):= {\displaystyle\lim_{t\rightarrow 0^+}  } \frac{1}{t}\left( \psi (x+th)-\psi(x)\right)$ is the directional derivative of $\psi$  at  $x$ in the direction $h$.  
In this case the  bifunctions $f$ and $g$ are obviously monotone and upper hemicontinuous.  Hence theorems on weak (resp. strong) convergence apply, whenever (\ref{fitz-discret}) (resp.  (\ref{fitz-discret}) and strong monotonicity) is satisfied.

%\noindent
{\bf - Weak convergence:}
Without any loss of generality we assume $\min_K \psi  = 0$.\\ Set $M= \underset{K}{\hbox{argmin}}\,\psi$, and consider $\overline\psi (x)=\psi (x)$ if $x\in K$, and  $\overline\psi (x)=+\infty$ if $x\notin K$; then $\overline\psi (x)- \delta_{M} (x)\leq 0$ for all $x \in H$. Using the reverse inequality for their Fenchel conjugates, we deduce $\overline\psi^* (p)- \sigma_{M} (p)\geq 0$ for all $p \in H$, and in view of Proposition \ref{fitz-phi}, condition \eqref{fitz-discret} becomes: $\forall u\in M, \;\text{for all}\;  p \in  {N}_{M}(u),$
\begin{equation}\label{acp-assump}
\sum_{n=1}^{+\infty}\lambda_n \beta_{n}\left[\overline\psi ^{*}\left(\frac{2p}{\beta_{n}}\right)-
\sigma_{M}\left(\frac{2p}{\beta_{n}}\right)\right]< +\infty.
\end{equation}
%which is simply the assumption originates from \cite{attouch2011prox}. 
%We refer to  for other variants and concrete examples for which this assumption is verified.\\
Applying Theorem \ref{thm-weak}, and supposing that $M$ is nonempty, $\displaystyle \liminf_{n\rightarrow +\infty} \lambda_n > 0$, $\displaystyle \lim_{n\rightarrow +\infty} \beta_n = +\infty$ and $0\leq \alpha< \frac{1}{3}$, then the whole sequence $\{ x_n\}$ generated by algorithm \eqref{algo} weakly converges to a point $\bar{x}$ solution of $(HMP)$.\\
\noindent Consider the particular case  where $\psi (x)=\dfrac{1}{2}d(x,M)^2$ and $M\subset K$ is a nonempty closed convex set satisfying $d(x,M)=\displaystyle \inf_{y\in M}\| x-y\|$. Then,  $\overline\psi^* (p)- \sigma_{M} (p)= \dfrac{1}{2}\| p\|^2$ for all $p \in H$. Here, $M$ is the minimum set of $\psi$, and then 
%\begin{center}
condition (\ref{fitz-discret}) is equivalent to $\displaystyle \sum_{n=1}^{+\infty}\dfrac{\lambda_n}{\beta_{n}}  < +\infty$.
%\end{center}

\begin{rem}
We note that the condition \eqref{acp-assump} is simply the assumption originates from \cite{attouch2011prox} in the framework of solving a variational inequality of the forme $$Ax+N_C(x)\ni 0,$$
where $A: H\rightrightarrows H$ is a maximally monotone operator and $C\subset H$ is a closed convex set. For this problem, the authors in \cite{attouch2011prox} obtained solutions by means of the convergence analysis of the trajectories of the following prox-penalization algorithm 
$$x_n=(I+\lambda_n(A+\beta_n\partial \psi))^{-1}x_{n-1},$$
where $\lbrace \beta_n \rbrace$ and $\lbrace \lambda_n \rbrace$ are two sequences of nonnegative reals and $\psi:H\rightarrow \R \cup \{+\infty\}$ acts as an external penalization function with respect to the constraint $x\in C$. Indeed, several ergodic and non ergodic convergence results have been justified for $\{x_n\}$ under the key assumption: $\;\text{for all}\;   p\in R(N_C)$,
\begin{equation*}
\sum_{n=1}^{+\infty}\lambda_n \beta_{n}\left[\psi ^{*}\left(\frac{p}{\beta_{n}}\right)-
\sigma_{C}\left(\frac{p}{\beta_{n}}\right)\right]< +\infty,
\end{equation*}
where $R(N_C)$ denotes the range of $N_C$.
\end{rem} 

{\bf - Strong convergence:} 
To deduce the strong convergence of the algorithm $(IPA)$ to a solution of $(HMP)$, we'll have to add a strong monotonicity condition on the function $g$. However, when we set $g(x,y)=\varphi (y)-\varphi (x)$, the strong monotonicity of $g$ is not assured, so that we take  $g(x,y)=\langle \nabla \varphi(x),y-x \rangle$, where $\nabla \varphi$ is the gradient of $\varphi$
(we identify $\varphi$ with $\overline\varphi (x) = \varphi (x) $ if $x\in K$, and $\overline\varphi (x) = +\infty $ if $x\notin K$).
In this case our inertial proximal scheme associated to the problem \eqref{pbep} is the following: $y_n:=x_n+\alpha(x_n-x_{n-1})$ and $x_{n+1}\in K$ {such that} 
\begin{equation}\label{syst-bil-prog}
\beta_n (\psi(y)-\psi(x_{n+1}))+\langle \nabla \varphi(x_{n+1}),y-x_{n+1} \rangle+\dfrac{1}{\lambda_n}\langle x_{n+1}-y_n, y-x_{n+1}\rangle \geq 0, \;\;\forall y\in K.
\end{equation}
Moreover, if we suppose $\varphi$ to be strongly convex on $K$, i.e., for some $\kappa>0$ and for all $x,y\in K$ and all $t\in [0,1]$ 
\begin{equation*}
\varphi (tx+(1-t)y)\leq t\varphi (x) + (1-t)\varphi (y) - \kappa t(1-t)\| x-y\|^2,
\end{equation*}
we deduce that $g$ is strongly monotone, and thus the conclusion of Theorem \ref{strong1} is valid whenever $\displaystyle \sum_{n=1}^{\infty}\lambda_n=+\infty$ and $0\leq \alpha< \frac{1}{3}$.\\
\vskip 4mm

%%%%%%%%%%%%%%%%%%%%%%%%%%%%%%%%%%%%%%%%%%%%%

\subsection{Equilibrium problem under a saddle point constraint}

Let $H_1, H_2$ be two real Hilbert spaces, $U\subset H_1$ and $V\subset H_2$ be nonempty closed convex sets, and let $L: U\times V \rightarrow \mathbb{R}$ be closed and convex-concave, i.e., for each $(u,v)\in U\times V$ the real
functions $L(.,v)$ and $-L(u,.)$ are convex and lower semicontinuous.

We consider the saddle-point problem: find $\;(\bar{u},\bar{v})\in U\times V \;$ such that
\begin{equation}\label{sp}
\tag{$SP$}
L(\bar{u},v)\leq L(\bar{u},\bar{v})\leq L(u,\bar{v}) \;\text{for every}\; (u,v)\in U\times V,  
\end{equation}
which is equivalent, see \cite{Ek-Tem}, to
$$
\displaystyle \max_{v\in V}\inf_{u\in U}L(u,v)=\min_{u\in U}\sup_{v\in V}L(u,v)=L(\bar{u},\bar{v}).
$$
Setting $H=H_1\times H_2$, $K=U\times V$, we define the bifunction $f: K\times K\rightarrow \mathbb{R}$ as:
\begin{equation*}\label{sp1}
f((u_1,v_1),(u_2,v_2)):=L(u_2,v_1)-L(u_1,v_2), \text{ for each }(u_1,v_1),(u_2,v_2)\in K.
\end{equation*}
Let us observe that problems $(SP)$ and $(EP)$ are equivalent and we denote the solution set of $(SP)$ by $S_{L}$. \\

Using the definition of the  Fitzpatrick transform $\mathcal{F}_f$, we have \\for all $(u_1,v_1),(u_2,v_2)\in K$:
\begin{align*}
\mathcal{F}_f((u_1,v_1),(u_2,v_2))&= \sup_{(x,y)\in K} \lbrace \langle u_2,x \rangle+\langle v_2,y \rangle+f((x,y),(u_1,v_1))\rbrace\\
&=\sup_{(x,y)\in K} \lbrace \langle v_2,y \rangle+L(u_1,y)-L(x,v_1)+\langle u_2,x \rangle\rbrace\\
&= \sup_{y\in V}\lbrace \langle v_2,y \rangle-(-L((u_1,y))\rbrace+\sup_{x\in U}\lbrace \langle u_2,x \rangle-L(x,v_1)\rbrace \\
&= (-L(u_1,.))^*(v_2)+(L(.,v_1))^*(u_2).
\end{align*}
Therefore the condition \eqref{fitz-discret} is satisfied when for all pairs $(u,v)\in S_f$ and $(p,q)\in N_{S_f}(u,v)$,
\begin{equation}\label{sig6}
%\tag{$\mathcal{C}_5$}
\sum_{n=1}^{+\infty} \lambda_n\beta_n\left[(-L(u,.))^*\left(\frac{2q}{\beta_n}\right)+(L(.,v))^*\left(\frac{2p}{\beta_n}\right)-\sigma_{S_f}\left(\frac{2p}{\beta_n},\frac{2q}{\beta_n}\right)\right]<+\infty.
\end{equation}
We consider two single-valued monotone operators $A$ and $B$ such that $K\subset\dom A\times \dom B$ and $A\times B+N_{S_L}$ is a maximally monotone operator (see \cite{roc,roc1}). Furthermore we suppose that  the solution set  $S_{VL}$ of $0\in A\bar x\times B\bar y+N_{S_L}(\bar x,\bar y)$ is nonempty. By $A\times B$, we mean the operator defined for $(u,v)\in H=H_1\times H_2$ by $(A\times B)(u,v)=Au\times Bv$. When the monotone
operator $A\times B+N_{S_L}$ is maximally monotone, then 
$$(\bar x,\bar y)\in  S_{VL} \quad \iff\quad  
	(\bar x,\bar y)\in  S_{L}\hbox{ and }\langle A\bar x, u-\bar x\rangle+\langle B\bar y, v-\bar y\rangle\geq 0, \;\forall (u,v)\in S_L.
	$$
For each $(u_1,v_1),(u_2,v_2)\in K$, let us set 
	$$g((u_1,v_1),(u_2,v_2)):=\langle Au_1, u_2-u_1\rangle+\langle Bv_1, v_2-v_1\rangle.$$
Then, our inertial proximal algorithm $(IPA)$ used for approaching a solution to the problem $(BEP)$ associated with the above bifunctions $f$ and $g$, i.e., for finding a solution in $S_{VL}$, takes the following form: for every $n\geq 1$,  given current iterates $(x_{n-1}^i, x_n^i)\in K$, $i=1, 2$, set $y_n^i=x_n^i+\alpha (x_n^i-x_{n-1}^i)$ and define $(x_{n+1}^1,x_{n+1}^2)\in K$ in this way:
\begin{equation}\label{EDSP}
%\tag{$EDSP$}  
\left\{ \begin{array}{l}\text{ for all } (u,v)\in U\times V, \\
\frac{1}{\lambda_n}\langle (x_{n+1}^1,x_{n+1}^2)-(y_n^1,y_n^2), (u,v)-(x_{n+1}^1,x_{n+1}^2)\rangle+\langle Ax_{n+1}^1, u-x_{n+1}^1\rangle\\ 
\hfill +\langle Bx_{n+1}^2, v-x_{n+1}^2\rangle+ \beta_n(L(u,x_{n+1}^2)-L(x_{n+1}^1,v)) \geq 0.
\end{array}\right.
\end{equation}
In this case,  Theorems \ref{thm-weak} and \ref{strong1} can be summarized as follows:
%%%%%%%%%%%%%%%%%
\begin{coro}\label{corsp} 
	Let $\{x_n^1,x_n^2\}$ be the sequence generated by \eqref{EDSP}. Under hypothesis \eqref{sig6} and whenever $0\leq \alpha< \frac{1}{3}$, $\displaystyle \liminf_{n\rightarrow +\infty}\lambda_n>0$ and  $\{\beta_n\} \rightarrow +\infty$, the
	weak convergence of $\{x_n^1,x_n^2\}$ to a solution of $S_{VL}$ is ensured. Also, the strong convergence of $\{x_n^1,x_n^2\}$ to the unique element of $S_{VL}$ is ensured when $0\leq \alpha< \frac{1}{3}$, $\displaystyle \sum_{n=1}^{\infty}\lambda_n=+\infty$ and  
	$A\times B$ is  strongly monotone on $K$.
\end{coro}
%We note that this result extends Theorem 2.1 and Theorem 2.2 in \cite{AC}, from monotone variational inequalities over the set   of solutions for convex minimization problem, to variational problems constrained by the saddle points of a convex-convave function.

Next, let us give an example where the condition \eqref{sig6} is verified. %for a non bilinear form.
%%%%%%%%%%%%%%%%%%%%%%%%%%%%%%%%%%%%%%
\begin{expl}
	Take $K=[0,1]\times [0,1]$ and $L$ the closed convex-concave function  defined on $K$ by $L(u,v)=u^2(1+v)$.
	Then, the set of saddle points of $L$, which is also the solution set $S_f$, is $S_f=\lbrace 0 \rbrace \times [0,1]$.\\ We also have
	\begin{align*}
	(p,q)\in N_{(\lbrace 0 \rbrace \times [0,1])}(0,v)&\Leftrightarrow (p,q)(s,t-v)\leq 0, \;\;\forall\;(s,t)\in \lbrace 0 \rbrace \times [0,1] \\
	&\Leftrightarrow q(t-v)\leq0,\;\;\forall t\in [0,1];
	\end{align*}
	and then 
	\begin{center}
		$N_{S_f}(0,0)=\mathbb{R}\times \mathbb{R}_-$, $N_{S_f}(0,1)=\mathbb{R}\times \mathbb{R}_+$ and 
		$N_{S_f}(0,v)=\mathbb{R}\times \lbrace 0 \rbrace$ for every $ v\in ]0,1[
		$.
	\end{center}
	To ensure \eqref{sig6}, we check
	\begin{center}$
		\begin{array}{rl} 
		\displaystyle\sigma_{S_f}\left(\frac{2p}{\beta_n},\frac{2q}{\beta_n}\right)&= \displaystyle\sup_{v\in  [0,1]} \left \lbrace v\frac{2q}{\beta_n}\right \rbrace = \left\{\begin{array}{cl} \displaystyle\frac{2q}{\beta_n} & \text{ if }  q>0\\
		0 & \text{ if }   q\leq 0,    
		\end{array}\right.
		\end{array}
		$ \end{center}
	
	\begin{center}$
		\begin{array} {rl}
		(-L(0,.))^*\left(\frac{2q}{\beta_n}\right)&=\displaystyle\sup_{0\leq s\leq 1}\left\lbrace \frac{2q}{\beta(t)}s\right\rbrace 
		=\left\{\begin{array}{cl} \displaystyle\frac{2q}{\beta_n} & \text{ if }  q>0\\
		0 & \text{ if }   q\leq 0,
		\end{array}\right.
		\end{array}
		$ \end{center}
	and 
	\begin{center}$
		\begin{array} {rl}
		(L(.,v))^*\left(\frac{2p}{\beta_n}\right)&=\displaystyle\sup_{0\leq s\leq 1}\left\lbrace \frac{2p}{\beta_n}s-(1+v)s^2\right\rbrace  
		=\frac{p^2}{(1+v)\beta_n^2}.
		\end{array}
		$\end{center}
	Thus
	\begin{center}$\begin{array}{l}
		\displaystyle \sum_{n=1}^{+\infty} \lambda_n\beta_n\left[(-L(u,.))^*\left(\frac{2q}{\beta_n}\right)+(L(.,v))^*\left(\frac{2p}{\beta_n}\right)-\sigma_{S_f}\left(\frac{2p}{\beta_n},\frac{2q}{\beta_n}\right)\right]
		\\= \frac{p^2}{(1+v)}{\displaystyle \sum_{n=1}^{+\infty} }\dfrac{\lambda_n}{\beta_n},
		\end{array}$
	\end{center}
	and then
	\eqref{sig6} is satisfied, whenever $\displaystyle \sum_{n=1}^{+\infty} \frac{\lambda_n}{\beta_n} 
	<+\infty$.
\end{expl}

\section{Numerical experiment}

In this section, we present a numerical experiment to illustrate the convergence of the proposed algorithm.
Let us consider the constrained  minimization problem $(HMP)$, with $$ K = \mathbb{R}^{2}, \quad \varphi(x)=\frac{1}{4}(x_1-x_2-2)^{2}\quad \text{ and } \quad \psi(x)=\frac{1}{4}\left(x_1+x_2-4\right)^{2}
.$$ Since $ \psi $ is convex, the minimum set of $\psi$ is $S_\psi =\nabla\psi^{-1} (0,0)=\{ x=(x_1,x_2)\in \mathbb{R}^{2}: x_2=4-x_1\}$ and the solution set of the hierarchical problem $\displaystyle \min_{S_\psi} \varphi$ is $S = \{\bar x\}=\{(3, 1)\}$. 
%with $\varphi (\overline x) = 14 $.\\

%To check condition \eqref{fitz-discret}, 

\noindent We evaluate  $ \frac{1}{2} d(x, S_\psi ) ^ 2 $ where $d (x, S_\psi )=\displaystyle \inf_{y\in S_\psi }\|y-x\|_2$ and $\|(x_1,x_2)\|_2:=\sqrt{x_1^2+x_2^2}$. For $ x = (x_1, x_2) \in \mathbb R ^ 2 $, we have
\begin{center}$
	d (x, S_\psi ) ^ 2 = \displaystyle \inf_{y_1 \in \mathbb R} \left((y_1-x_1)^2 + (y_1 + x_2-4)^2 \right) = \displaystyle \inf_{y_1 \in \mathbb R} \alpha (y_1).
	$ \end{center}
Since $\alpha (t)=(t-x_1)^2 + (t + x_2-4)^2$ is strongly convex and 
\begin{center} 
	$\alpha^\prime (\bar {y}_1) = 2 (2 \bar {y}_1 + x_2-4-x_1) = 0 \Leftrightarrow \bar{y}_1 = \frac {1} {2} (x_1-x_2 + 4)$,  \end{center}
we get 
\begin{center}
	$
	d (x, S_\psi) ^ 2 = \alpha (\bar{y}_1) = (\bar{y}_1-x_1)^2 + (\bar{y}_1 + x_2-4 )^2 = 2 \psi (x),
	$ \end{center}
which yields $ \psi (x) = \frac{1}{2} d (x, S_\psi)^2 $. Thus condition \eqref{fitz-discret} is equivalent to $\displaystyle\sum_{n=1}^{+\infty}\dfrac{\lambda_n}{\beta_{n}}  < +\infty.$\\
%The associated bifunctions are defined for all $x,y\in K$ by $f(x,y)=\psi(y)-\psi(x)$  and $g(x,y)=\varphi(y)-\varphi(x).$ Let us take $F_n(x,y)= g(x,y)+\beta_nf(x,y),$ since the proximal operator is the resolvent of $F_\beta(x,y)=F,$ which is also a resolvent of the gradient
%of $\varphi+\beta_n\psi. $ 
%Then for $\lambda > 0$ and $x \in K$  we have
%$$ F_{\lambda}^{F_n} (x) = prox_{\lambda}^{\varphi+\beta_n\psi} (x)= prox_{\lambda}^{\theta_n\varphi}\left(\theta_nx+4(1-\theta_n)\right)=\qquad \mbox{with}\quad \theta_n=\frac{1}{1+\frac{\beta_{n}}{2}} $$
%We have $$prox_{\lambda}^{\theta_n\varphi}\left(\theta_nx+4(1-\theta_n)\right)=\left(I+\lambda_n\theta_n B\right)^{-1}\left(\theta_nx+4(1-\theta_n)\right)$$

Note that the associated bifunctions are defined for all $x,y\in K$ by $f(x,y)=\psi(y)-\psi(x)$ and $g(x,y)=\varphi(y)-\varphi(x)$, that $f$ and $g$ are monotone and that weak and strong convergences coincide in finite dimension. 
%, so that we only need to verify strong convexity of 
%$\varphi$, which is assured since its Hessian is positive.\\
%and $\bar{x} = (3,1)$ is the unique solution of problem \eqref{bil-prog}.

%Note that the objective function $ \varphi $ is strongly convex, since its Hessian is positive, because the corresponding  two eigenvalues $ 3 + \sqrt 2$ and $ 3- \sqrt 2 $ are positive. So it is more convenient to take $ g (x, y) = \langle \nabla \varphi (x), y-x \rangle $ in order to get the strong monotonicity of $ g $. In this case, when $\displaystyle\int_0^{+\infty} \dfrac{dt}{\beta(t)}<+\infty$, we use Theorem \ref{theo2} to conclude the convergence of $ x (t) $ to the unique solution $ \bar{x} = (1,3) $ of $\left( {\mathcal M}\right)$.\\

By using the proximal operator of $\varphi+\beta_n\psi,$ the drawing in Figure 1 displays the asymptotic behavior of the trajectories
$x_n = ( y_n,z_n )$  from the initial values $(y_0 , z_0 )=(0,0.5)$ and  $(y_1 , z_1 )=(0,0.5)$ with $\alpha = 0.1, \lambda_n=\frac{1}{n}$ and different values of $\beta_n.$ We also use the iterate error $\|x_n-\bar{x}\|_2$ as a measure to describe the computational performance of our algorithm. The numerical results in Figure 2 illustrate the rate of convergence of $\|x_n-\bar{x}\|_2$  for different choices of $\beta_n$ and $\alpha = 0.1,$ while Figure 3 displays the convergence rate of $\|x_n-\bar{x}\|_2$ for different choices of $\alpha$ and $\beta_n=(1+n).$

\begin{figure}[H]
	\centering
	\label{fig:a}\includegraphics[scale=0.5]{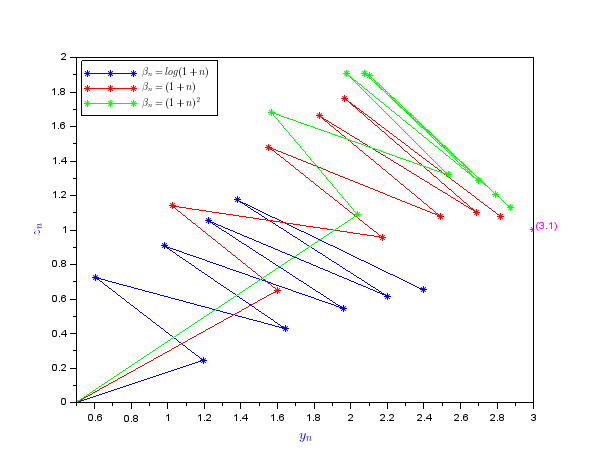}
	\caption{The asymptotic behavior of the trajectories
		$x_n = ( y_n,z_n )$.}
\end{figure}

\begin{figure}[H]
	\centering
	\label{fig:b}\includegraphics[scale=0.5]{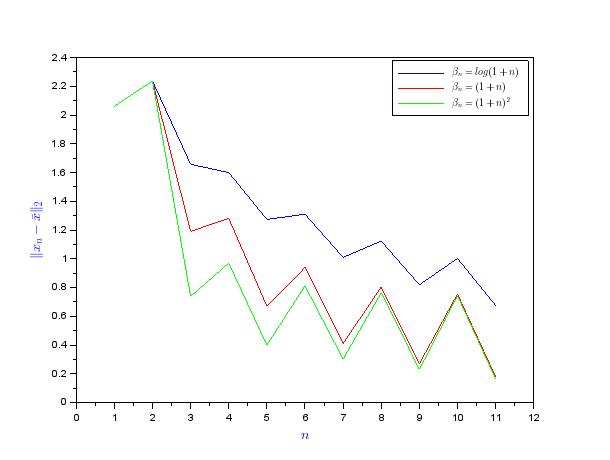}
	\caption{The rate of convergence of $\|x_n-\bar{x}\|_2$  for  $\alpha = 0.1$.}
\end{figure}

\begin{figure}[H]
	\centering
	\label{fig:c}\includegraphics[scale=0.5]{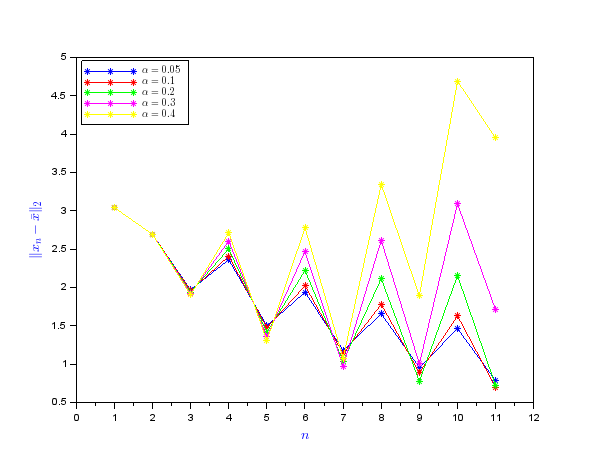}
	\caption{The convergence rate of $\|x_n-\bar{x}\|_2$ for  $\beta_n=(1+n).$}
\end{figure}

We note in Figure 2, that when $\beta_n $ increases then the rate of convergence of  $\|x_n-\bar{x}\|_2$ rapidly increases to $0,$ while in Figure 3, the constant coefficient $\alpha$ acts inversely on the speed of convergence of $\|x_n-\bar{x}\|_2,$ (the convergence gets worst as the values of $\alpha$ exceed $\frac{1}{3}$), which confirms the importance of taking $\alpha <\frac{1}{3}$ in our theoretical results. \\

We note that all codes in this digital test are written in SCILAB-6.1.

%%%%%%%%%%%%%%%%%%%%%%%%%%%%%%%%%%%%%%%%%%%%%%%%%%%%%%%%%%%%%%%%%%%%%%%%%%%%%%%%%%%%%
\section{Concluding Remark}
In this paper, we presented an inertial proximal \\
method for solving bilevel  monotone equilibrium problems in Hilbert spaces. Our analysis shows the weak and  the strong convergence of the trajectory generated by the algorithm under natural assumptions.
Our results can be seen as an extension and improvement of some known results in the literature. 
In particular, the geometric assumption \eqref{fitz-discret} shows that, as conjectured in \cite{moud2}, the restrictive assumption $\|x_{n+1}-x_n\|=o(\epsilon_n)$ may be removed via the introduction of a  notion of conditioning for equilibrium bifunctions.
%which is here the equilibrium Fitzpatrick transform $\mathcal{F}_f$ associated with the lower level equilibrium bifunction $f$. 
We illustrate this assumption with two concrete particular cases
and conclude this work by a numerical experiment,
%to hierarchical minimization problems 
which shows that, with a suitable choice of the parameters, the convergence conditions are satisfied and the proposed iterative method succeeds in approximating a solution to bilevel equilibrium problems. 

Finally, we note that, to the best of our knowledge, our approach seems to be the first introduced inertial proximal scheme for solving $(BEP)$ and then
several extensions of our main results may be analyzed. In particular, 
an interesting direction of future research will be to obtain the above weak convergence result without condition \eqref{fitz-discret} and also to develop new splitting inertial proximal algorithms for solving bilevel equilibrium problems. \\

%\paragraph{\textbf{Acknowledgments}:} The research of A\"icha Balhag was supported by  the EIPHI Graduate School (contract ANR-17-EURE-0002). Research of Michel Th\'era benefited from the support of the FMJH Program PGMO and from the support of EDF.
%\bibliographystyle{plain}
\bibliographystyle{siamplain}
\bibliography{Omega}

\begin{thebibliography}{10}

\bibitem{ait2021dynamical}
{\sc M.~Ait~Mansour, Z.~Mazgouri, and H.~Riahi}, {\em A dynamical approach for
  the quantitative stability of parametric bilevel equilibrium problems and
  applications}, Optimization, 71 (2022), pp.~1389--1408,
  \url{https://doi.org/10.1080/02331934.2021.1981892}.

\bibitem{AH}
{\sc M.~H. Alizadeh and N.~Hadjisavvas}, {\em On the {F}itzpatrick transform of
  a monotone bifunction}, Optimization, 62 (2013), pp.~693--701,
  \url{https://doi.org/10.1080/02331934.2011.653975}.

\bibitem{AA}
{\sc F.~Alvarez and H.~Attouch}, {\em An inertial proximal method for maximal
  monotone operators via discretization of a nonlinear oscillator with
  damping}, Set-Valued Var. Anal., 9 (2001), pp.~3--11,
  \url{https://doi.org/10.1023/A:1011253113155}.

\bibitem{ant}
{\sc A.~S. Antipin}, {\em Convergence and estimates for the rate of convergence
  of proximal methods to fixed points of extremal mappings}, Zh. Vychisl. Mat.
  i Mat. Fiz., 35 (1995), pp.~688--704.

\bibitem{attouch2011prox}
{\sc H.~Attouch, M.~Czarnecki, and J.~Peypouquet}, {\em Prox-penalization and
  splitting methods for constrained variational problems}, SIAM J. Optim., 21
  (2011), pp.~149--173, \url{https://doi.org/10.1137/100789464}.

\bibitem{Att-Riahi-Thera}
{\sc H.~Attouch, H.~Riahi, and M.~Th{\'e}ra}, {\em Somme ponctuelle
  d'op\'{e}rateurs maximaux monotones}, Serdica Math. J., 22 (1996),
  pp.~267--292.

\bibitem{Bauchk}
{\sc H.~H. Bauschke and P.~L. Combettes}, {\em Convex analysis and monotone
  operator theory in {H}ilbert spaces}, CMS Books in Mathematics/Ouvrages de
  Math\'{e}matiques de la SMC, Springer, New York, 2011,
  \url{https://doi.org/10.1007/978-1-4419-9467-7}.

\bibitem{bento-2016}
{\sc G.~C. Bento, J.~X. Cruz~Neto, J.~O. Lopes, P.~A. Soares, Jr., and
  A.~Soubeyran}, {\em Generalized proximal distances for bilevel equilibrium
  problems}, SIAM J. Optim., 26 (2016), pp.~810--830,
  \url{https://doi.org/10.1137/140975589}.

\bibitem{BO}
{\sc E.~Blum and W.~Oettli}, {\em From optimization and variational
  inequalities to equilibrium problems}, Math. Student, 63 (1994),
  pp.~123--145.

\bibitem{BC1}
{\sc R.~I. Bo\c{t} and E.~R. Csetnek}, {\em Forward-backward and {T}seng's type
  penalty schemes for monotone inclusion problems}, Set-Valued Var. Anal., 22
  (2014), pp.~313--331, \url{https://doi.org/10.1007/s11228-014-0274-7}.

\bibitem{BC2}
{\sc R.~I. Bo\c{t}, E.~R. Csetnek, and N.~Nimana}, {\em An inertial
  proximal-gradient penalization scheme for constrained convex optimization
  problems}, Vietnam J. Math., 46 (2018), pp.~53--71,
  \url{https://doi.org/10.1007/s10013-017-0256-9}.

\bibitem{BG}
{\sc R.~I. Bo\c{t} and S.~Grad}, {\em Approaching the maximal monotonicity of
  bifunctions via representative functions}, J. Convex Anal., 19 (2012),
  pp.~713--724.

\bibitem{bns}
{\sc H.~Br\'{e}zis, L.~Nirenberg, and G.~Stampacchia}, {\em A remark on {K}y
  {F}an's minimax principle}, Boll. Un. Mat. Ital. (4), 6 (1972), pp.~293--300.

\bibitem{bk}
{\sc R.~Burachik and G.~Kassay}, {\em On a generalized proximal point method
  for solving equilibrium problems in {B}anach spaces}, Nonlinear Anal., 75
  (2012), pp.~6456--6464, \url{https://doi.org/10.1016/j.na.2012.07.020}.

\bibitem{CCR}
{\sc O.~Chadli, Z.~Chbani, and H.~Riahi}, {\em Equilibrium problems with
  generalized monotone bifunctions and applications to variational
  inequalities}, J. Optim. Theory Appl., 105 (2000), pp.~299--323,
  \url{https://doi.org/10.1023/A:1004657817758}.

\bibitem{CMR}
{\sc Z.~Chbani, Z.~Mazgouri, and H.~Riahi}, {\em From convergence of dynamical
  equilibrium systems to bilevel hierarchical {K}y {F}an minimax inequalities
  and applications}, Minimax Theory Appl., 4 (2019), pp.~231--270.

\bibitem{CR2}
{\sc Z.~Chbani and H.~Riahi}, {\em Variational principles for monotone and
  maximal bifunctions}, Serdica Math. J., 29 (2003), pp.~159--166,
  \url{https://doi.org/10.4099/math1924.29.159}.

\bibitem{j}
{\sc Z.~Chbani and H.~Riahi}, {\em Weak and strong convergence of an inertial
  proximal method for solving {K}y {F}an minimax inequalities}, Optim. Lett., 7
  (2013), pp.~185--206, \url{https://doi.org/10.1007/s11590-011-0407-y}.

\bibitem{CR1}
{\sc Z.~Chbani and H.~Riahi}, {\em Weak and strong convergence of
  prox-penalization and splitting algorithms for bilevel equilibrium problems},
  Numer. Algebra Control Optim., 3 (2013), pp.~353--366,
  \url{https://doi.org/10.3934/naco.2013.3.353}.

\bibitem{Cotrina}
{\sc J.~Cotrina, M.~Th\'{e}ra, and J.~Z\'{u}\~{n}iga}, {\em An existence result
  for quasi-equilibrium problems via {E}keland's variational principle}, J.
  Optim. Theory Appl., 187 (2020), pp.~336--355,
  \url{https://doi.org/10.1007/s10957-020-01764-0}.

\bibitem{Dempe}
{\sc S.~Dempe}, {\em Annotated bibliography on bilevel programming and
  mathematical programs with equilibrium constraints}, Optimization, 52 (2003),
  pp.~333--359, \url{https://doi.org/10.1080/0233193031000149894}.

\bibitem{Ek-Tem}
{\sc I.~Ekeland and R.~Temam}, {\em Convex analysis and variational problems},
  Studies in Mathematics and its Applications, Vol. 1, North-Holland Publishing
  Co., Amsterdam-Oxford; American Elsevier Publishing Co., Inc., New York,
  1976.
\newblock Translated from the French.

\bibitem{Hieu-Gibali}
{\sc D.~V. Hieu and A.~Gibali}, {\em Strong convergence of inertial algorithms
  for solving equilibrium problems}, Optim. Lett., 14 (2020), pp.~1817--1843,
  \url{https://doi.org/10.1007/s11590-019-01479-w}.

\bibitem{mosco}
{\sc U.~Mosco}, {\em Implicit variational problems and quasi variational
  inequalities}, in Nonlinear operators and the calculus of variations
  ({S}ummer {S}chool, {U}niv. {L}ibre {B}ruxelles, {B}russels, 1975), Lecture
  Notes in Math., Vol. 543, Springer, Berlin, 1976, pp.~83--156.

\bibitem{moud3}
{\sc A.~Moudafi}, {\em Proximal point algorithm extended to equilibrium
  problems}, J. Nat. Geom., 15 (1999), pp.~91--100.

\bibitem{moud1}
{\sc A.~Moudafi}, {\em Second-order differential proximal methods for
  equilibrium problems}, JIPAM. J. Inequal. Pure Appl. Math., 4 (2003),
  pp.~Article 18, 7.

\bibitem{moud2}
{\sc A.~Moudafi}, {\em Proximal methods for a class of bilevel monotone
  equilibrium problems}, J. Global Optim., 47 (2010), pp.~287--292,
  \url{https://doi.org/10.1007/s10898-009-9476-1}.

\bibitem{moud-thera}
{\sc A.~Moudafi and M.~Th\'{e}ra}, {\em Proximal and dynamical approaches to
  equilibrium problems}, in Ill-posed variational problems and regularization
  techniques ({T}rier, 1998), vol.~477 of Lecture Notes in Econom. and Math.
  Systems, Springer, Berlin, 1999, pp.~187--201,
  \url{https://doi.org/10.1007/978-3-642-45780-7\_12}.

\bibitem{opial}
{\sc Z.~Opial}, {\em Weak convergence of the sequence of successive
  approximations for nonexpansive mappings}, Bull. Amer. Math. Soc., 73 (1967),
  pp.~591--597, \url{https://doi.org/10.1090/S0002-9904-1967-11761-0}.

\bibitem{Riahi5}
{\sc H.~Riahi}, {\em On the maximality of the sum of two maximal monotone
  operators}, Publ. Mat., 34 (1990), pp.~269--271,
  \url{https://doi.org/10.5565/PUBLMAT\_34290\_05}.

\bibitem{clr}
{\sc H.~Riahi, Z.~Chbani, and M.-T. Loumi}, {\em Weak and strong convergences
  of the generalized penalty {F}orward-{F}orward and {F}orward-{B}ackward
  splitting algorithms for solving bilevel hierarchical pseudomonotone
  equilibrium problems}, Optimization, 67 (2018), pp.~1745--1767,
  \url{https://doi.org/10.1080/02331934.2018.1490957}.

\bibitem{roc}
{\sc R.~T. Rockafellar}, {\em On the maximal monotonicity of subdifferential
  mappings}, Pacific J. Math., 33 (1970), pp.~209--216,
  \url{http://projecteuclid.org/euclid.pjm/1102977253}.

\bibitem{roc1}
{\sc R.~T. Rockafellar}, {\em On the maximality of sums of nonlinear monotone
  operators}, Trans. Amer. Math. Soc., 149 (1970), pp.~75--88,
  \url{https://doi.org/10.2307/1995660}.

\bibitem{Thuy-2017}
{\sc L.~Q. Thuy and T.~N. Hai}, {\em A projected subgradient algorithm for
  bilevel equilibrium problems and applications}, J. Optim. Theory Appl., 175
  (2017), pp.~411--431, \url{https://doi.org/10.1007/s10957-017-1176-2}.

\end{thebibliography}
\end{document}